\documentclass[preprint, 11pt, authoryear]{elsarticle}

\usepackage{amssymb}
\usepackage{amsmath}
\usepackage{amsthm}
\newtheorem{definition}{Definition}[section]
\newtheorem{remark}{Remark}[section]
\newtheorem{proposition}{Proposition}[section]
\newtheorem*{proposition*}{Proposition}

\newtheorem{lemma}{Lemma}[section]
\newtheorem{example}{Example}[section]

\usepackage{graphicx}
\usepackage{comment}
\usepackage{mathrsfs}
\usepackage{multirow}
\usepackage{booktabs}
\usepackage{xcolor}
\usepackage{tikz}
\usetikzlibrary{hobby}
\usetikzlibrary{shapes, fit, positioning, backgrounds}
\usepackage{hyperref}
\usepackage{enumitem}
\usepackage[normalem]{ulem}
\usepackage{todonotes}
\usepackage{subcaption}
\usepackage[margin=2cm]{geometry}
\usepackage{setspace}
\usepackage{rotating}

\usepackage[nofillcomment,noend,linesnumbered]{algorithm2e}

\SetCommentSty{mycommfont}
\RestyleAlgo{ruled}

\newcommand{\hip}{{\mathcal H}}
\newcommand{\VV}{{\mathcal V}}
\newcommand{\EE}{{\mathcal E}}
\newcommand{\PP}{{\mathcal P}}
\newcommand{\QQ}{{\mathcal Q}}

\newcommand{\N}{{\mathbb N}}
\newcommand{\R}{{\mathbb R}}
\newcommand{\Z}{{\mathbb Z}}
\newcommand{\CG}[1]{G(#1)}

\newcommand{\STAB}{{\mathscr S}}
\newcommand{\card}[1]{\left| #1 \right|}
\newcommand{\cn}[1]{\chi(#1)}

\newcommand{\cfcn}[1]{\chi_{CF}(#1)}
\newcommand{\cfcnT}[1]{\chi_{CF}^{T}(#1)}
\newcommand{\dpcn}[1]{\chi_{DP}(#1)}
\newcommand{\newproblem}{\text{DPCP}}
\newcommand{\setcoveringform}{\text{DPCP-SCF}}
\newcommand{\setcoveringformLR}{\text{LR}}
\newcommand{\setcoveringformLRM}{\text{LR}^{\text{M}}_{\text{ext}}}
\newcommand{\setcoveringformLRext}{\text{LR}_{\text{ext}}}
\newcommand{\pricingform}{\text{DPCP-PRI}}

\journal{European Journal of Operation Research}

\begin{document}

\begin{frontmatter}



\title{
Double-Partition Coloring problem: a unified approach for conflict-free coloring of hypergraphs and other coloring problems
in graphs
}
%% Article title


\author[UNR,CONICET]{Mauro Lucci\corref{cor}}
\ead{mlucci@fceia.unr.edu.ar}

\author[UNR,CONICET]{Graciela Nasini}
\ead{nasini@fceia.unr.edu.ar}

\author[UNR,CONICET]{Paola Tolomei}
\ead{ptolomei@fceia.unr.edu.ar}

\author[EPN]{Luis Miguel Torres}
\ead{luis.torres@epn.edu.ec}

\cortext[cor]{Corresponding author}

\affiliation[UNR]{organization={Depto. de Matemática (FCEIA), Universidad Nacional de Rosario},
            addressline={Pellegrini 250}, 
            city={Rosario},
            postcode={2000}, 
            state={Santa Fe},
            country={Argentina}}
\affiliation[CONICET]{organization={CONICET},
            addressline={Ocampo y Esmeralda}, 
            city={Rosario},
            postcode={2000}, 
            state={Santa Fe},
            country={Argentina}}
\affiliation[EPN]{organization={MODEMAT, Escuela Politécnica Nacional},
            addressline={Av. Ladrón de Guevara E11-253}, 
            city={Quito},
            postcode={170109}, 
            state={Pichincha},
            country={Ecuador}}

%% Abstract
\begin{abstract}

The Double-Partition Coloring Problem (DPCP) is a recently introduced generalization of several coloring problems on graphs and hypergraphs, including the Partition Coloring Problem, the List Coloring Problem, and the Conflict-Free Coloring Problem (CFCP). In this work, we further investigate the structural and computational properties of the DPCP and develop exact and heuristic solution approaches. We derive structural results that allow the detection of infeasible instances and the reduction of their size. We propose one-step and two-step heuristics for obtaining high-quality solutions. These procedures are incorporated into an enhanced branch-and-price algorithm based on a set covering formulation of the DPCP. The algorithm further includes new strategies for solving the NP-hard pricing problem. Computational experiments show that the proposed branch-and-price algorithm and a compact integer programming formulation complement each other: while the compact formulation is generally more effective on sparse instances, branch-and-price performs particularly well on denser and larger instances. In particular, these are the first exact algorithms to be computationally evaluated on various CFCP instances.
\end{abstract}


% %Research highlights
% \begin{highlights}
% \end{highlights}

%% Keywords

\end{frontmatter}



\section{Introduction}
\label{sec.introduction}

The \emph{Double-Partition Coloring Problem} (DPCP) was recently introduced by \cite{LucciEtAl2025} as a unifying framework for several graph coloring problems.
Given a graph $G=(V,E)$ and two partitions $\PP$ and $\QQ$ of $V$, the objective is to %\quitar{assign colors to some vertices so} 
find a partial proper coloring such that exactly one vertex in each part of $\PP$ is colored, and all colored vertices within each part of $\QQ$ receive the same color. 
%\quitar{, and no two adjacent vertices of $V$ receive the same color,} \quitar{while minimizing the total number of colors used}. 
The goal is to minimize the total number of colors used.

Several well-known generalizations of the classical Vertex Coloring Problem (VCP) %\quitar{graph coloring problems} 
are polynomial-time reducible to the DPCP, including the Partition Coloring Problem (PCP) \citep{li2000partition}, the List Coloring Problem (LCP) \citep{Lucci2024}, and the Conflict-free Coloring Problem on hypergraphs (CFCP) \citep{Even2003}. %\quitar{, which are well-known generalizations of the classical vertex coloring problem (VCP)}.
All these problems %are considered in their optimization versions, where the objective is to minimize the number of colors used, and are formally defined in the next section. Graph coloring problems 
arise in a wide range of applications related to resource allocation in areas such as scheduling, compilation, and telecommunications \citep{JensenToft2011}.
Exact algorithms for solving general instances have been developed, where \emph{branch-and-price} (B\&P) schemes represent the state of the art. See \citet{Trick1996,MalagutiEtAl2011} for VCP,  \citet{Furini2018} for PCP, and \citet{Lucci2024} for LCP.
In contrast, to the best of our knowledge, prior to \citet{LucciEtAl2025} the CFCP had neither been studied through integer linear programming (ILP) nor provided with exact algorithms for general instances.

In the CFCP, the objective is to color a subset of the vertices of a hypergraph, while ensuring that each hyperedge contains at least one vertex whose color is not assigned to any other vertex within that hyperedge.
This problem arises from an application in frequency assignment in telecommunications, where frequency channels from a finite spectrum must be assigned to antennas so that they can operate without interference. 
%\quitar{A classical motivation for the VCP arises from frequency assignment in telecommunications.} 
A common modeling approach represents antennas as vertices of a graph and uses edges to indicate that the corresponding antennas have intersecting coverage areas. 
If frequencies are represented by colors, the task becomes to color the vertices so that adjacent vertices receive different colors,
that is, a VCP instance.
However, this model may be overly restrictive, since an intersection of coverage areas does not necessarily imply interference, which may lead to an unnecessarily large number of frequencies.
For instance, suppose that the intersection of the coverage areas of two antennas is entirely contained within the coverage area of a third antenna. Clients located in that region may connect to the latter antenna. In this case, the first two antennas may share the same frequency, provided that it differs from the frequency assigned to the third antenna. More generally, to avoid interference it suffices that a user located at any point within the network coverage area can access at least one antenna transmitting on a unique frequency. This requirement naturally leads to the CFCP.

Figure~\ref{fig:ex-three-antennas} illustrates the example proposed in \cite{Smorodinsky13}. Antennas $a_1$, $a_2$, and $a_3$ have mutually intersecting coverage areas, as shown on the left, which under the standard VCP-based approach requires assigning a distinct frequency to each antenna. Observe, however, that the coverage area of the network can be partitioned into six regions, each characterized by the set of antennas that can be reached from that region:
$E_1 = \{a_1\},\ E_2 = \{a_2\},\ E_3 = \{a_3\},\ E_4 = \{a_1, a_2\},\ E_5 = \{a_2, a_3\},\ E_6 = \{a_1, a_2, a_3\}$.
Consider the hypergraph, depicted on the right, whose vertices correspond to the three antennas and whose hyperedges are given by the sets above. If one frequency is assigned to $a_1$ and $a_3$ and another to $a_2$, then each hyperedge contains a vertex whose frequency is unique within that hyperedge. For example, a client located in the region corresponding to hyperedge $E_6$ can communicate through antenna $a_2$, whose frequency differs from those assigned to $a_1$ and $a_3$.

\begin{figure}
\centering
\includegraphics[]{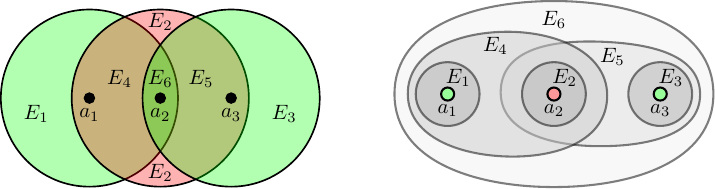}
\caption{An example of frequency assignment for three antennas with circular coverage areas, modeled as a CFCP. %(\cite{Smorodinsky13})
}
\label{fig:ex-three-antennas}
\end{figure}

The CFCP was originally studied on hypergraphs associated with various geometric settings \cite{HarPeledSmorodinsky03,ChenEtAl09,PachTardos10}. 
%\quitar{ These hypergraphs are usually defined in one of two possible ways: (i) The set of vertices of the hypergraph represents a family $\mathcal{R}$ of geometric regions in $\R^d$ (balls, boxes, half-spaces, or Jordan regions) and each hyperedge corresponds to a set of regions in $\mathcal{R}$ having nonempty intersection; or, (ii) given a family $\mathcal{R}$ of regions in $\R^d$ and a finite set of points $P \subset \R^d$, the points are the vertices of the hypergraph and each hyperedge is the subset of $P$ contained in one region from $\mathcal{R}$. 
Later, \cite{PachTardos09} studied the conflict-free coloring problem on general hypergraphs and on hypergraphs induced by (open or closed) neighborhoods of vertices of a graph. The authors refer to the latter problem as the \emph{conflict-free coloring of a graph}.
%For example, the authors show that the vertices of any graph with maximum degree $\Delta$ can be colored with $\log^{2 + \varepsilon} \Delta$ colors in such a way that the neighborhood of each vertex contains a uniquely colored vertex. 
More recent contributions have focused on conflict-free coloring of particular graph classes such as claw-free graphs, interval graphs (\cite{BhyravarapuEtAl22}), planar graphs (\cite{FabriciEtAl23}), graphs with bounded clique width, and intersection graphs (\cite{BhyravarapuEtAl24}). The extension of CFCP to list conflict-free coloring has also been addressed (\cite{Gupta2026}). 

% To the best of our knowledge, the CFCP has not been addressed thorough integer linear programming (ILP), and no exact algorithms are known for general instances. In contrast, branch-and-price (B\&P) schemes represent the state of the art for solving the VCP} \cite{Trick1996, MalagutiEtAl2011}, \quitar{the PCP} \cite{Furini2018}, and the LCP} \cite{Lucci2024}.

% The authors of \cite{LucciEtAl2025} present two integer linear programming formulations for the DPCP: a \textit{compact formulation} based on decision variables modeling color-to-vertex assignments, and a \textit{set covering formulation} with an exponential number of variables representing stable sets of vertices that receive the same color. In line with the trend mentioned above, the preliminary computational results reported in their work suggest that a branch-and-price (B\&P) approach applied to the set covering formulation is also promising for the DPCP, potentially providing the first practical exact method for solving CFCP instances. 

The contributions of this work are twofold.
First, we advance the structural study of the DPCP by deriving new results that enable the detection of infeasible instances and the reduction of instance size.
Moreover, we provide a formal proof of the polynomial-time reductions from CFCP and LCP to DPCP, which was previously stated without proof in \cite{LucciEtAl2025}. 
These results establish new connections between coloring problems in graphs and hypergraphs.
Second, we strengthen the B\&P algorithm previously proposed for the DPCP, contributing to the development of practical exact methods for solving the CFCP.
The improved algorithm integrates the new theoretical insights to preprocess instances, introduces specifically designed heuristics for the DPCP, and incorporates new methods for solving the pricing problem.
Moreover, general DPCP instances are computationally evaluated for the first time using this enhanced algorithm.
%\textcolor{magenta}{¿Decir algo de las instancias de CFCP?} \textcolor{blue} { 
In particular, this evaluation includes a diverse set of CFCP instances.
%}.
%\quitar{ The contributions of this work are twofold and build upon the results of} \cite{LucciEtAl2025}. \quitar{First, we advance the structural study of the DPCP by establishing decomposition results, characterizing feasible instances, and formally proving the reduction of several graph coloring problems to the DPCP. Second, we develop an enhanced B\&P algorithm for the DPCP that provides a practical exact method for solving CFCP instances. The algorithm incorporates the new theoretical insights, introduces heuristics for both the DPCP and the pricing problem, and is evaluated through an extensive computational study on synthetic DPCP and CFCP instances.}

The remainder of this paper is organized as follows. In the next section, required notation is introduced and the coloring problems mentioned above are formally defined. Their relationship with the DPCP is studied in Section~\ref{sec:definition:dpcp}. Section~\ref{sec:feasibility:DPCP} presents fundamental structural properties of the DPCP, including a characterization of feasible instances.
Heuristic approaches for solving the DPCP are presented in Section~\ref{sec:heuristics}. The set covering integer linear programming formulation proposed in \cite{LucciEtAl2025} is revisited in Section~\ref{sec:branch-and-price-algo}, where the main components of a B\&P algorithm are described. 
Section \ref{sec:B&Pconfigurations} describes the experimental setup and the preliminary experiments used to configure the B\&P algorithm.
Results of computational experiments conducted both on general DPCP instances and on instances stemming from the CFCP are discussed in Section~\ref{sec:computationalresults}. Finally, Section~\ref{sec:conclusions} presents concluding remarks and directions for future research.

\section{Notation and definitions}\label{sec.preliminaries}

For $n \in \N$, in the following we denote by $[n]$ the set  $\{z\in \N: z\leq n\}$. 

Given a graph $G=(V,E)$, a set $K \subseteq V$ of mutually adjacent vertices of $G$ is called a \emph{clique} in $G$, and a set $S \subseteq V$ of mutually non adjacent vertices is known as a \emph{stable set} or \emph{independent set} in $G$. The \emph{stability number} or \emph{independence number} $\alpha(G)$ is the largest cardinality of a stable set in $G$.
%, i.e., $\alpha(G)= \max\{ \card{S} : S \mbox{ stable set in } G\}$.
For a vertex $v \in V$, $N(v)$ denotes the set of its neighbors, $N[v]=N(v)\cup \{v\}$, and $G-v$ is the graph obtained by removing $v$. For a subset $U \subseteq V$, $G[U]$ denotes the subgraph of $G$ induced by $U$.
    
A \textit{hypergraph} $\hip$ is a pair $(\VV, \EE)$ where $\VV$ is a finite set of elements called vertices and $\EE$ is a family of non-empty subsets of $\VV$ called hyperedges. Graphs are the special class of hypergraphs where all hyperedges have cardinality two, in which case they are simply referred to as edges.

A \emph{coloring} of a hypergraph (or a graph) $(\VV, \EE)$ is an assignment $c: \VV' \to C$ of a subset $\VV' \subseteq \VV$ of its vertices to a set $C$ of colors. The vertices in $\VV'$ are the \emph{colored} vertices. The coloring is \emph{total} if  $\VV'=\VV$. 
%otherwise it is \emph{partial} \textcolor{red}{¿es otherwise o todo total es partial?}. 
Coloring problems on graphs and hypergraphs have been extensively studied in the literature (see, e.g.~\cite{JensenToft2011,MalagutiToth2010} for two prominent survey works). 
We introduce the coloring problems considered throughout this work.

The classical \emph{Vertex Coloring Problem} asks for the minimum number of colors required to define a total \emph{proper} coloring on a graph $G$, i.e., a total coloring such that no two adjacent vertices receive the same color. This minimum number is called the \emph{chromatic number} $\cn{G}$ of $G$. 
%Several integer linear programming models for the VCP have been proposed in the literature. The first, commonly referred to as the \emph{descriptive} model, includes binary decision variables for assigning colors to vertices as well as variables indicating whether a color is used. \cite{MendezDZabala2006,MendezDZabala2008} study this formulation and propose a Branch-and-Cut algorithm that incorporates additional inequalities to strengthen the linear relaxation and to handle symmetry issues. An alternative formulation is the \emph{set covering} formulation, which involves an exponential number of variables corresponding to stable sets in $G$ that represent feasible color classes. \cite{Trick1996}, and later,  \cite{MalagutiEtAl2011} propose Branch-and-Price approaches to solve this formulation. 
%\textcolor{blue}{Currently, this decomposition approach is widely recognized as the most efficient exact method for the VCP and its various generalizations. Its superiority has been demonstrated in recent studies for the partition coloring problem and the list coloring problem, as well as in our own computational experiments where the Branch-and-Price algorithm consistently outperformed compact integer programming formulations.}

In the \emph{List Coloring Problem}, a list $L_v \subseteq C$ of admissible colors is given for each $v \in V$. The task is to find a total proper coloring that satisfies $c(v) \in L_v$ for every $v \in V$ and minimizes the number of required colors \citep{Lucci2024}.
The feasible solutions of LCP are termed as \emph{list colorings}. Note that LCP is a generalization of VCP. 
Indeed, if $L_v = C$ holds for all $v \in V$, then list colorings are exactly the total proper colorings of $G$. 
%A Branch-and-Price algorithm for a set covering formulation of LCP was proposed by  \cite{Lucci2024}.
%Moreover, DPCP generalizes LCP, as it is shown at next. 

A coloring $c: \VV' \to C$ of a hypergraph $\hip=(\VV, \EE)$, with $\VV' \subseteq \VV$, is called \textit{conflict-free} if each hyperedge $E \in \EE$ contains a vertex of \emph{unique color}, i.e., a vertex $v \in E \cap \VV'$ satisfying $\card{\{ w \in E \, : \, c(w) = c(v) \}} = 1$. The \textit{Conflict-Free Coloring Problem}  is the task to determine a conflict-free coloring of $\hip$ minimizing the number of required colors \citep{Even2003}.
Its optimum value is known as the \textit{conflict-free chromatic number} of $\hip$ and is denoted by $\cfcn{\hip}$. 
The total conflict-free chromatic number $\cfcnT{\hip}$ is the minimum number of colors required to define a total conflict-free coloring of $\hip$. Observe that $\cfcn{\hip} \leq \cfcnT{\hip}$ holds trivially. Moreover, $\cfcnT{\hip} \leq \cfcn{\hip}+1$, as any %partial 
conflict-free coloring of $\hip$ can be extended to a total conflict-free coloring by assigning a new color to each uncolored vertex of $\hip$.
%\textcolor{red}{¿Notamos además que $\cfcn{\hip} \leq \cfcnT{\hip}$?}
%\textcolor{blue}{Lo he agregado, aunque creo que es trivial.}
Total CFCP is a generalization of VCP, as it is straightforward to verify that 
both problems coincide on graphs. Indeed,
any total proper coloring of a graph $G$ is a total conflict-free coloring, as each edge of $G$ contains at least one (actually two) uniquely colored vertices. Conversely, if $uv \in E$, then any total conflict-free coloring of $G$ assigns two different (unique) colors to the vertices $u$ and $v$. Hence, $\cfcnT{G} = \cn{G}$ holds for any graph $G$.

Given a partition $\PP = \{P_i: i \in  [k]\}$ of the set of vertices of a graph $G=(V, E)$, the \emph{Partition Coloring Problem} asks for a proper 
%(partial or total) 
coloring of $G$ such that exactly one vertex of each partition class is colored, while the number of required colors is minimized \citep{li2000partition}. %The optimal objective value $\chi_P(G, \PP)$ of this problem is the \emph{partition chromatic number} of $G$ with respect to the partition $\PP$. 
Observe that VCP is the special case of PCP obtained when $\PP = \{\{v\} : v \in V\}$. %At the same time, DPCP generalizes PCP, as the latter is obtained when $\PP_1 = \PP, \PP_2=\{\{v\}: v\in V\}$.
% \cite{Furini2018} consider a set covering formulation for PCP with an exponential number of variables and propose a Branch-and-Price solution algorithm.

\section{The Double-Partition Coloring Problem}
\label{sec:definition:dpcp}

Given a graph $G=(V,E)$, and two partitions  $\PP$ and $\QQ$ of $V$, a \emph{double-partition coloring} of $G$ is a proper coloring of $G$ that satisfies the following two constraints:
\begin{enumerate}[label=(C\arabic*)]
    \item For each $P \in \PP$, exactly one vertex of $P$ is colored. \label{def:DPCP:C1}
    \item For each $Q \in \QQ$, all colored vertices in $Q$ are assigned the same color. \label{def:DPCP:C2}
\end{enumerate}
The \emph{Double-Partition Coloring Problem} (DPCP) is the optimization problem of finding the minimum number of colors required to define a double-partition coloring on $G$, referred to as its \emph{double-partition chromatic number} $\dpcn{G, \PP, \QQ}$. Formally, \\

\noindent
\fbox{%
\begin{minipage}{0.98\linewidth}
\textbf{Double-Partition Coloring Problem (DPCP)}\\
\textsc{instance}. A graph $G=(V,E)$ and two partitions 
$\PP$ and $\QQ$ of $V$.\\
\textsc{objective}. To obtain $\dpcn{G, \PP, \QQ}$.
\end{minipage}%
}

\bigskip

In the remainder of this paper  $(G, \PP, \QQ)$ denotes an instance of DPCP with $\PP =\{P_i: i \in [n]\}$ and $\QQ =\{Q_j: j \in [m]\}$. 
Given a vertex $v$, we denote by $i(v) \in [n]$ (resp. $j(v) \in [m]$) the unique index such that $v \in P_{i(v)}$ (resp. $v \in Q_{j(v)}$).

\begin{remark}\label{clique-estable}
%Let $(G, \PP, \QQ)$ be an instance of DPCP. 
For each $P \in \PP$, adjacencies among vertices in $P$ impose no restriction on the coloring, since exactly one vertex in P is assigned a color. 
This observation generalizes Proposition~2 in \cite{li2000partition} for the PCP. 
Therefore, throughout this paper, each set P is assumed to be either a clique or an independent set, whichever is more convenient for the arguments.
\end{remark}

No reductions are known between LCP, PCP, and CFCP. However, we will show that each of these problems can be reduced to specific instances of DPCP.

% Observe that VCP is the special case of DPCP obtained when $\PP= 
% \QQ=\{\{v\}: v\in V\}$. {\color{red} Yo aca no ondria por que generaliza al VSP porque ya esta en otros papers como cada uno de los tres que vamos a mencionar lo generalizan}
Clearly, PCP is obtained from DPCP when $\QQ=\{\{v\}: v\in V\}$.
%In fact, it also holds for every partition $\QQ$ verifying that, for all $Q\in \QQ$ there exists $P\in \PP$ such that $Q\subseteq P$.
In fact, the same holds for any partition $\QQ$ satisfying that each part $Q \in \QQ$ is contained in some part $P \in \PP$.

\color{black}
% In the following  we show that DPCP also generalizes the other two coloring problems defined in the previous section, a fact that was {\color{red} already? }mentioned in \cite{LucciEtAl2025}.

Moreover,  consider an instance of LCP defined by a graph $G=(V,E)$, a color set $C$, and, for each $v \in V$, a set $L_v \subseteq C$ of admissible colors. Define $\tilde G = (\tilde V, \tilde E)$ as the graph with vertex set $\tilde V = \{(v,k) : v \in V,\ k \in L_v\}$, and edge set $\tilde E$ such that $(v_1,k_1)(v_2,k_2) \in \tilde E$ if and only if  $v_1v_2 \in E$ or $k_1 \neq k_2$. Let      
    \begin{align*}
    \PP &= \{P_v : v \in V\},  \text{ with } P_v=\{(v',k) \in \tilde{V}: v'=v\}, \mbox{ and } \\ 
    \QQ &= \{Q_k  : k \in C\}, \text{ with } Q_k = \{(v,k') \in \tilde{V}: k'=k \} 
    \end{align*}
    be two partitions of $\tilde{V}$. 
Figure~\ref{fig:ej-lcp2dpcp} illustrates this construction. 
Then, the following result can be proved:

\begin{figure}
\begin{subfigure}{0.3\textwidth}
\centering
\includegraphics{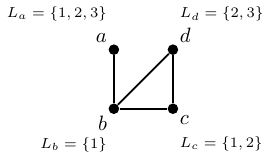}
\end{subfigure}
\begin{subfigure}{0.69\textwidth}
\centering
\includegraphics{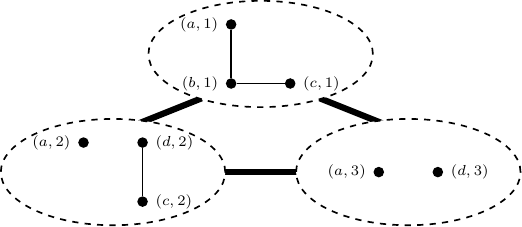}
\end{subfigure}
\caption{Example of the construction of $\tilde G$ (right) from graph $G$ (left). Each subgraph on the right contains the vertices of a part $Q_k$, $k\in C$. Each thick segment represents the complete join between the corresponding vertex sets.}
\label{fig:ej-lcp2dpcp}
\end{figure}

\begin{proposition}
Each list coloring of $G$ induces a double-partition coloring of $(\tilde G, \PP, \QQ)$ with the same number of colors, and conversely. 
\end{proposition}

\begin{proof} 
    Let $c: V \longrightarrow C$ be a list coloring of $G$ and  $\tilde V' = \{(v,c(v)) : v \in V\} \subseteq \tilde{V}$. 
    Consider the color assignment $\tilde c: \tilde V' \to C$ defined by $\tilde c(v,c(v)) = c(v)$.
    Clearly, $\tilde c$ verifies constraints \ref{def:DPCP:C1} and \ref{def:DPCP:C2}, and uses as many colors as $c$.
    To prove that $\tilde c$ is a double-partition coloring of $\tilde G$, it remains to show that $\tilde c$ is a 
    %partial 
    proper coloring of $\tilde G$. Indeed, let $(v_1,k_1)$ and $(v_2,k_2)$ be two adjacent colored vertices of $\tilde G$. By the definition of $\tilde E$, either $v_1v_2 \in E$ or $k_1 \neq k_2$. Moreover, observe that $\tilde c(v_1,k_1) = k_1 = c(v_1)$ and $\tilde c(v_2,k_2) = k_2 = c(v_2)$. Thus, if $v_1v_2 \in E$, then $c(v_1) \neq c(v_2)$ as $c$ is a proper coloring of $G$. Otherwise, $k_1 \neq k_2$ must hold.
    In both cases, it follows that $\tilde c (v_1,k_1) \neq \tilde c(v_2,k_2)$.

    For the converse, let $\tilde c$ be a double-partition coloring for $\tilde G$.
    Due to constraint \ref{def:DPCP:C1}, for each $v\in V$ there is a unique $k_v \in L_v$ such that the vertex $(v, k_v) \in P_v$ is colored. 
    Observe that two vertices $(v_1, k_{v_1})$ and $(v_2, k_{v_2})$ have the same color if and only if $k_{v_1} = k_{v_2}$. Indeed, if $k_{v_1} = k_{v_2}$ then $(v_1, k_{v_1}), (v_2, k_{v_2}) \in Q_{k_{v_1}} = Q_{k_{v_2}}$ and constraint \ref{def:DPCP:C2} requires both vertices to have the same color. Otherwise, if $k_{v_1} \neq k_{v_2}$, the vertices are adjacent and must have different colors, as $\tilde c$ is proper. As a consequence, the number of colors required by $\tilde c$ is equal to $|\{k_v : v \in V\}|$.
    
    Now consider the color assignment $ c: V \longrightarrow C$ defined by $c(v)=k_v \in L_v$, which clearly requires the same number of colors as $\tilde c$.
    To prove that $c$ is a proper coloring of $G$, let $v_1,v_2$ be two adjacent vertices from $V$. 
    By the definition of $\tilde E$, $(v_1,k_{v_1})$ and $(v_2,k_{v_2})$ are adjacent in $\tilde G$. 
    Hence, $\tilde c(v_1,k_{v_1})\neq \tilde c(v_2,k_{v_2})$, as $\tilde c$ is a proper coloring of $\tilde G$. 
    Moreover, $k_{v_1} \neq k_{v_2}$ follows from the observation above and thus $c(v_1) \neq c(v_2)$.
\end{proof}

%Using similar ideas, one can show that the Double-Partitioning Coloring Problem also generalizes the Conflict-Free Coloring Problem. 
To clarify the relationship between DPCP and CFCP, we first introduce the notion of the \emph{conflict graph of a hypergraph}.

\begin{definition}
\label{def:3.1}
Given a hypergraph $\hip = (\VV, \EE)$ its \textit{conflict graph $\CG{\hip}$} is the graph whose vertex set is 
$$
V(\CG{\hip})=\left\{(E,v): E \in \mathcal \EE, v \in E\right\},
$$
and whose edge set is defined as follows:
$$
(E_1,v_1) (E_2,v_2) \in E(\CG{\hip}) \Leftrightarrow v_1 \neq v_2 \mbox{ and } %E_1\cap E_2\cap \{v_1,v_2\}\neq \emptyset.
(v_1 \in E_2 \mbox{ or } v_2 \in E_1)
$$
%The neighborhood $N((E,v))$ of a vertex $(E,v) \in V(\CG{\hip})$ is therefore given by 
%$$N((E,v))=\{(\tilde E, w): w\neq v  \text{ and }  \{v,w\} \cap E\cap \tilde E\neq \emptyset\}. $$ 
\end{definition}

The idea underlying the definition above is that assigning a color $k$ to a vertex $(E,v)$ of $\CG{\hip}$ encodes two decisions in the hypergraph $\hip$: assigning color $k$ to vertex $v$ in hyperedge $E$, and forbidding any other vertex of $E$ from receiving the same color $k$.
Consequently, two vertices $(E_1,v_1)$ and $(E_2,v_2)$ of $\CG{\hip}$ with $v_1 \neq v_2$ cannot be assigned the same color whenever $v_2 \in E_1$ or $v_1 \in E_2$. Therefore, such vertices must be adjacent in $\CG{\hip}$.
Moreover, consider the two partitions of $V(\CG{\hip})$ defined by:
\begin{align*}
\PP &= \{ P_E: E \in \EE \}, \text{ with } P_E = \{(E,v') \in V(G(\hip)) : v' \in E\}, \mbox{ and}\\
\QQ &= \{ Q_v: v \in V \}, \text{ with } Q_v = \{(E',v) \in V(G(\hip)) : v \in E'\}.
\end{align*}
By Definition \ref{def:3.1}, $P_E$ is a clique for each $E \in \mathcal{E}$, since for any two distinct vertices $(E,v_1),(E,v_2) \in P_E$, we have $v_1\in E$ and $v_2\in E$.
An example of this construction is illustrated in Figure~\ref{fig.hypergraph}, adapted from an example presented in \cite{LucciEtAl2025}.
Then, the following result can be proved.

\begin{figure}
\centering
\begin{subfigure}{0.22\textwidth}
\centering
\includegraphics{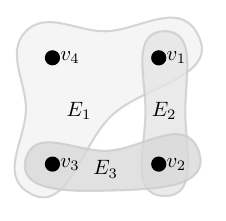}
\caption{Hypergraph $\hip$.}
\end{subfigure}
\hfill
\begin{subfigure}{0.33\textwidth}
\centering
\includegraphics{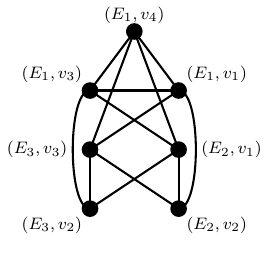}
\caption{Conflict graph $G(\hip)$.}
\end{subfigure}
\hfill
\begin{subfigure}{0.33\textwidth}
\centering
\includegraphics{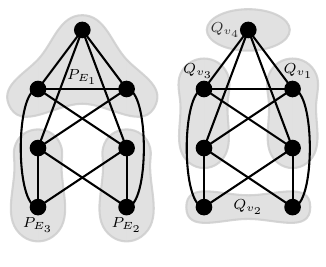}
\caption{Partitions $\PP$ and $\QQ$ of $V(G(\hip))$.}
\end{subfigure}
\caption{
Example of the construction of the conflict graph $G(\hip)$ for the hypergraph $\hip$.}
\label{fig.hypergraph}
\end{figure}

\begin{proposition}\label{prop.cfcp.as.dpcp}
Each conflict-free coloring of $\hip$ induces a double-partition coloring of $(\CG{\hip}, \PP, \QQ)$ with the same number of colors, and conversely. Therefore, $$
\cfcn{\hip}=\dpcn{\CG{\hip}, \PP, \QQ}.
$$

\end{proposition}

\begin{proof}
Let $c:\VV' \to C$, with $\VV' \subseteq \VV$, be a conflict-free coloring of $\hip$. For each $E \in \EE$, let $v_E \in E$ be a uniquely colored vertex in $E$ with respect to $c$. Define 
$
V' = \{(E,v_E) : E \in \EE\} \subseteq V(\CG{\hip}),
$
and let $\tilde c:V' \to C$ be the coloring of $\CG{\hip}$ defined by $\tilde c(E,v_E) = c(v_E)$.
Clearly, $\tilde c$ uses no more colors than $c$. 
Moreover, by construction, $\tilde c$ satisfies conditions~\ref{def:DPCP:C1} and~\ref{def:DPCP:C2}.

We now show that $\tilde c$ is a proper coloring of $\CG{\hip}$. 
Let $(E_1,v_{E_1})$ and $(E_2,v_{E_2})$ be adjacent vertices in $V'$. 
By Definition~\ref{def:3.1}, we have $v_{E_1} \neq v_{E_2}$ and either $v_{E_2} \in E_1$ or $v_{E_1} \in E_2$. 
Assume, without loss of generality, that $v_{E_2} \in E_1$. 
Since $v_{E_1}$ is the unique vertex of $E_1$ colored with $c(v_{E_1})$, it follows that $c(v_{E_1}) \neq c(v_{E_2})$, and therefore $\tilde c(E_1,v_{E_1}) \neq \tilde c(E_2,v_{E_2})$. Thus, $\tilde c$ is proper.

For the converse direction, let $\tilde c:\tilde{V}' \to C$, with $\tilde{V}'  \subseteq V(\CG{\hip})$, 
be a double-partition coloring of $(\CG{\hip},\PP,\QQ)$. 
For each $v \in \VV$, define
$Q'_v = Q_v \cap \tilde{V}' $,
and let $\mathcal{W} = \{ v \in \VV : Q'_v \neq \emptyset \}$.
By condition~\ref{def:DPCP:C2}, all vertices in $Q'_v$ receive the same color. 
Hence the mapping $c:\mathcal{W} \to C$ defined by
$c(v) = \tilde c(E,v), \quad \text{for any } (E,v) \in Q'_v$,
is a well-defined coloring of $\hip$ and uses the same number of colors as $\tilde c$.

Finally, we prove that $c$ is conflict-free. In fact, condition~\ref{def:DPCP:C1}states that, for any $E \in \EE$, exactly one vertex $(E,v_E) \in P_E$ is colored. 
Thus, $v_E$ is colored in $c$. 
Suppose there exists another colored vertex $v' \in E$, $v' \neq v_E$. Then $(E',v') \in \tilde{V}'$ holds for some $E' \in \EE$,  $E' \neq E$. Moreover,  observe that
$(E,v_E)$ and $(E',v')$ are adjacent in $\CG{\hip}$.
Since $\tilde c$ is proper, $\tilde c(E,v_E) \neq \tilde c(E',v')$, and therefore $c(v_E) \neq c(v')$.
Hence, $v_E$ has a color distinct from all other colored vertices in $E$, 
and $c$ is a conflict-free coloring of $\hip$.
\end{proof}

As a consequence of Proposition~\ref{prop.cfcp.as.dpcp}, DPCP generalizes CFCP. Moreover, this generalization is strict, since all CFCP instances are feasible, while there are infeasible instances of DPCP, as we shall see in the next section.

\section{First results on particular  instances of DPCP} \label{sec:feasibility:DPCP}

 In this section, we analyze DPCP instances with special structures that yield preprocessing rules, which either reduce the problem to smaller equivalent instances or detect infeasibility.

Recall that $(G,\PP,\QQ)$ denotes an instance of the DPCP, with $G=(V,E)$, $\PP = \{P_i : i \in [n]\}$, and $\QQ = \{Q_j : j \in [m]\}$.
Due to Remark~\ref{clique-estable}, we assume that each set in $\PP$ induces a clique in $G$. 

Let us first introduce the following remark for instances where $n=1$ or $m=1$.

\begin{remark} \label{nm_igual1}

Instances $(G,\PP,\QQ)$ with $\PP = \{V\}$ are always feasible and $\dpcn{G, \PP, \QQ} = 1$. Indeed, assigning an arbitrary color to any vertex $v \in V$ yields an optimal double-partition coloring. For instances $(G,\PP,\QQ)$ with $\QQ = \{V\}$, since all vertices must share the same color and each $P \in \PP$ is a clique, $(G,\PP,\QQ)$ is feasible if and only if $\alpha(G) = n$. 
In this case, all colored vertices belong to the same part $Q=V$ and, by \ref{def:DPCP:C2}, must therefore receive the same color. Hence, exactly one color is used, and $\dpcn{G, \PP, \QQ} = 1$.
\end{remark}

From the previous remark, in the following we consider DPCP instances where $n$ and $m$ at least two.

%\medskip 

Observe that, if
%is a singleton, say 
$P = \{v\}$ for some part $P \in \PP$, then $v$ must always be colored and this leads to the following possible instance reductions:

\begin{lemma}
\label{th:singletons}
Given $(G,\PP,\QQ)$, let $P\in\PP$ such that $P=\{u\}$ and $v \in  Q_{j(u)}\setminus \{u\}$.

\begin{itemize}
\item If $v\in N(u)$ and $P_{i(v)}=\{v\}$, then $(G,\PP,\QQ)$ is infeasible. 

\item If $v\in N(u)$ and  $|P_{i(v)}|\geq 2$,  then solving $(G,\PP,\QQ)$ is equivalent to solving the instance $(G-v,\tilde \PP, \tilde \QQ)$, where 
$
\tilde \PP= \{P\setminus \{v\}: P\in \PP\}$ and $\tilde \QQ= \{Q\setminus \{v\}: Q\in \QQ\}$.

\item If $v\notin N(u)$ and $P_{i(v)}=\{v\}$ then solving $(G,\PP,\QQ)$ is equivalent to solving $(\tilde  G,\tilde\PP,\tilde \QQ)$    where $\tilde G$ is the graph obtained by contracting $u$ and  $v$, i.e. $\tilde G=(V\setminus \{v\}, \tilde  E)$, with $\tilde E=(E\setminus \{vw: w\in N(v)\}) \cup \{uw: w\in N(v)\})$,  
%{\color{red} ESTO ES LO QUE ESTA DESPUES PARA LEMA 4.1.
 %given a subset of vertices $W$, we define the graph $G[W]/\QQ$ with a vertex $w_j$ for each $j \in [m]$ such that $W \cap Q_j \neq \emptyset$. Two vertices $w_{j_1}$ and $w_{j_2}$ are adjacent in $G[W]/\QQ$ if and only if there exist $v_1 \in W \cap Q_{j_1}$ and $v_2 \in W \cap Q_{j_2}$ such that $v_1v_2 \in E$.
%ME PARECE DEJARLO COMO ESTA
%}
 $\tilde \PP=\PP\setminus \{\{v\}\}$ and $\tilde \QQ= \{Q\setminus \{v\}: Q\in \QQ\}$.
        
\end{itemize}
\end{lemma}

\begin{proof}
Since $P=\{u\}$, node $u$ must be colored in order to satisfy condition \ref{def:DPCP:C1}. Moreover, since $v\in Q_{j(u)}$, if $v$ is colored, then it must be assigned to the same color as $u$. 

Therefore, if $v\in N(u)$, nodes $u$ and $v$ cannot both be colored, and hence $v$ must remain uncolored.
    
If, in addition, $P_{i(v)}=\{v\}$, then $v$ must be colored to satisfy condition \ref{def:DPCP:C1} for $P_{i(v)}$, which contradicts the fact that $v$ must remain uncolored. Thus, the instance is infeasible.

If $|P_{i(v)}|\geq 2$, the fact that $v$ must remain uncolored implies that $v$ can be removed from the graph and from all sets in $\PP$ and $\QQ$. The resulting instance is equivalent to $(G,\PP,\QQ)$, proving the second statement.

Finally, suppose that $v\notin N(u)$ and $P_{i(v)}=\{v\}$. In this case, $v$ must be colored, and, since $v\in Q_{j(u)}$, both $u$ and $v$ must be assigned to the same color. Hence, $u$ and $v$ can be identified, yielding an equivalent instance obtained by contracting $u$ and $v$ as described in the third statement.
\end{proof}

Observe that, by iteratively applying this last result, we arrive at either an infeasible instance or an instance where, for every vertex $v$ such that $|P_{i(v)}| = 1$, $Q_{j(v)}$ is a stable set.

\medskip 

Many graph coloring problems can be decomposed according to the connected components of the underlying graph or hypergraph. 
For example, instances of the VCP, PCP, and CFCP can be solved independently on each connected component, and the resulting subsolutions can be merged into a complete solution. 
However, DPCP instances do not follow this pattern when some part in 
$\QQ$ intersects more than one connected component. 
Observe that this cannot occur for any part $P \in \PP$, since $P$ induces a clique and is therefore contained in a single connected component.
%This situation is illustrated in the following example.

\begin{example}
Figure~\ref{fig:nocon}$(a)$ shows an optimal double-partition coloring with 3 colors for an instance $(G,\PP,\QQ)$, where the parts of $\PP$ are represented by dashed lines and those of $\QQ$ are shaded. Note that $G$ has two connected components, $G_1$ and $G_2$, and each part $P\in\PP$ is fully contained in either $V(G_1)$ or $V(G_2)$. This, however, does not hold for all parts of $\QQ$.
Figure~\ref{fig:nocon}$(b)$ and $(c)$ show optimal double-partition colorings with 2 colors for the subinstances $(G_1,\PP^1,\QQ^1)$ and $(G_2,\PP^2,\QQ^2)$ induced by the connected components $G_1$ and $G_2$, respectively. Merging these colorings does not yield a feasible double-partition coloring for the original instance, since $v_2$ and $v_6$ receive different colors even though they belong to the same part of $\QQ$.

\begin{figure}
\begin{subfigure}{0.395\textwidth}
\centering
\includegraphics[width=\textwidth]{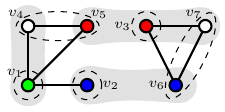}
\caption{$(G,\PP,\QQ)$}
\end{subfigure}
\hfill
\begin{subfigure}{0.22\textwidth}
\centering
\includegraphics[width=\textwidth]{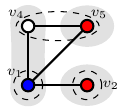}
\caption{$(G_1,\PP^1,\QQ^1)$}
\end{subfigure}
\hfill
\begin{subfigure}{0.22\textwidth}
\centering
\includegraphics[width=\textwidth]{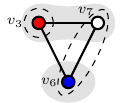}
\caption{$(G_2,\PP^2,\QQ^2)$}
\end{subfigure}
\caption{(a) An optimal double-partition coloring of $(G,\PP,\QQ)$; (b) and (c) optimal double-partition colorings of the subinstances $(G_1,\PP^1,\QQ^1)$ and $(G_2,\PP^2,\QQ^2)$, respectively.}
\label{fig:nocon}
\end{figure}

\end{example}

However, decomposition in connected components can be exploited to obtain lower bounds on the double-partition chromatic number of a graph. Indeed, the restriction of a double-partition coloring of $(G,\PP,\QQ)$ to the vertex set $V_k$ of a connected component $G_k$, yields a double-partition coloring of  $(G_k,\PP^k,\QQ^k)$, 
where  $\PP^k=\{P\cap V_k : P\in \PP, \,  P\cap V_k \neq \emptyset\}$ and $\QQ^k=\{Q\cap V_k : Q\in \QQ,\,  Q\cap V_k \neq \emptyset\}$.
Therefore,
$\dpcn{G_k,\PP^k,\QQ^k} \leq \dpcn{G,\PP,\QQ}$.

If each part of $\QQ$ intersects exactly one connected component of $G$, then the instance admits the following decomposition.

\begin{lemma}
\label{connectedcomponentsandisolatedvertices}
Let $G$ be a graph with $\ell$ connected components, each denoted by $G_k=(V_k, E_k)$ with $k\in [\ell]$, such that, for every $Q \in \QQ$, $Q \subseteq V_k$ for some $k \in [\ell]$.
For every $k \in [\ell]$, let $\PP^k = \{P \cap V_k: P \in \PP, \,P\cap V_k \neq \emptyset\}$ and $\QQ^k = \{Q \cap V_k: Q \in \QQ, \, Q\cap V_k \neq \emptyset\}$.
Then, $(G,\mathcal{P},\mathcal{Q})$ is feasible if and only if $(G_k, \mathcal{P}^k, \mathcal{Q}^k)$ is feasible for every $k\in [\ell]$. Moreover, if $(G,\mathcal{P},\mathcal{Q})$ is feasible, then
$\dpcn{G, \mathcal{P}, \mathcal{Q}}=\max \{\dpcn{G_k, \mathcal{P}^k, \mathcal{Q}^k}:k\in [\ell]\}$.
\end{lemma}

The proof follows immediately from the fact that in this case a feasible double-partitioning coloring for $(G, \PP, \QQ)$ can be obtained by merging the colorings of the components.

If $G$ has a connected component consisting of an isolated vertex $v$, then removing $v$ preserves feasibility and does not change the optimal value, regardless of whether $Q_{j(v)}$ intersects different connected components.

\begin{lemma}\label{Isolated}
Let $G$ be a graph with an isolated vertex $v$ and $|V| \geq 2$.
Define $G' = G \setminus \{v\}$, $\PP' = \PP \setminus \{\{v\}\}$, and $\QQ' = \{ Q \setminus \{v\} : Q \in \QQ\}$.
Then, $(G,\mathcal{P},\mathcal{Q})$ is feasible if and only if $(G', \mathcal{P}', \mathcal{Q}')$ is feasible.
Moreover, if $(G', \mathcal{P}', \mathcal{Q}')$ is feasible,
$\dpcn{G, \mathcal{P}, \mathcal{Q}} = \dpcn{G', \mathcal{P}', \mathcal{Q}'}$.
\end{lemma}

The proof is straightforward, since any optimal double-partition coloring of $(G', \mathcal{P}', \mathcal{Q}')$ can be extended to an optimal coloring of $(G, \mathcal{P}, \mathcal{Q})$ by assigning to $v$ the color used in $Q_{j(v)}$, if any vertex in $Q_{j(v)}$ is colored, or any color used in the solution, otherwise.

% \begin{lemma} \label{componentesconexasyaislados}
% Let $(G,\mathcal P, \mathcal Q)$ be an instance such that $G$ has at least two connected components, one of them denoted by $G'=(W, E_1)$, and let $G_2 = (V \setminus W, E \setminus E_1)$.
% Suppose that $\PP$ can be bipartitioned into $\{\PP_1, \PP_2\}$, where $\PP_1$ is a partition of $W$ and $P_2$ is a partition of $V \setminus W$ (and similarly for $\QQ$).
% % is a connected component of $G$ such that, for all $j\in [m]$, $Q_j\subseteq W$ or $Q_j\subseteq V\setminus W$,  $\mathcal P_1 = \{P_i\cap W: i\in [n], P_i\cap W\neq \emptyset \}$,  $\mathcal Q_1 = \{Q_j: j\in [m], Q_j \subset W\}$,  $\mathcal P_2 = \{P_i\cap V\setminus W: i\in [n], P_i\cap V\setminus W\neq \emptyset \}$, and  $\mathcal Q_2 = \mathcal Q\setminus \mathcal Q_1$. 
% Then, $(G,\mathcal P, \mathcal Q)$ is feasible if and only both $(G_1,\mathcal P_1, \mathcal Q_1)$ and $(G_2,\mathcal P_2, \mathcal Q_2)$ are feasible. 
% Moreover, if $(G,\mathcal P, \mathcal Q)$ is feasible, 
% $$\dpcn{G, \PP, \QQ}=\max\{\dpcn{G_1, \PP_1, \QQ_1}, \dpcn{G_2, \PP_2, \QQ_2}\}.$$
% In particular, if $W$ is a single isolated vertex and $(G_2,\mathcal P_2, \mathcal Q_2)$ is feasible, then
% $$\dpcn{G, \PP, \QQ}=\dpcn{G_2, \PP_2, \QQ_2}.$$ 
% \end{lemma}

\medskip

Finally we analyze feasibility of DPCP instances.
For the VCP, PCP, and CFCP, feasibility of any instance is guaranteed, provided that the number of colors is sufficiently large; however, this does not hold for the LCP.
In fact, deciding feasibility for the LCP is NP-complete \citep{golovach2017survey}, which in turn implies the NP-completeness of the feasibility problem for the DPCP. The smallest infeasible instance of the DPCP is given by the graph $G=(V,E)$ with $V=\{u,v\}$, $E=\{uv\}$, together with the partitions $\PP=\{\{u\},\{v\}\}$ and $\QQ=\{\{u,v\}\}$.

We define a \emph{set of representatives of} $(G,\PP,\QQ)$  as any subset $W \subseteq V$ 
such that $|W \cap P| = 1$ for each $P \in \PP$. 
Note that any double-partition coloring of $(G,\PP,\QQ)$ induces a set $W$ 
of representatives (namely, he colored nodes) and a proper coloring 
$c$ of $G[W]$. 

Unlike the PCP, given an arbitrary set $W$ of representatives, a proper coloring of the graph $G[W]$ cannot necessarily be extended to a double-partition coloring of $(G, \PP, \QQ)$. This is because the DPCP additionally requires all vertices 
belonging to the same $Q\in \QQ$ to be assigned the same color. Consequently, we must 
ensure that the intersection of $W$ with each $Q\in \QQ$ is a stable set in $G$. A set of representatives satisfying the latter condition will in the following be called a \emph{feasible set of representatives}.

%Based on these observations, we obtain the following feasibility characterization for an instance:
%\begin{lemma}\label{lemma.feasible.solutions}

\begin{remark}\label{lemma.feasible.solutions}
The set of colored vertices of any double-partition coloring of $(G,\PP,\QQ)$ is a feasible set of representatives of $(G,\PP,\QQ)$. Conversely, if $W$ is a feasible set of representatives,  assigning one color for each set $W \cap Q$, $Q \in \QQ$, yields a double-partition coloring of $(G,\PP,\QQ)$. Therefore, 
$(G,\PP,\QQ)$ is feasible if and only if %there exists a set of representatives $W$ such that $W\cap Q$ is a stable set of $G$ for every $Q \in \QQ$. 
there exists a feasible set of representatives.
\end{remark}

% From Remark  \ref{lemma.feasible.solutions},  the double partition chromatic number is the minimum of $\chi(G[W])$ taken over all feasible sets of representatives $W$. The following lemma shows that 
% \quitar{
% to compute $\chi(G[W])$ we can work  on a smaller graph obtained by contracting the vertices of each $Q_j \cap W$ in $G[W]$ into a single vertex $w_j$.} 
% {\color{red} we can actually work on coloring a graph a smaller than $G[W]$.}

To compute $\dpcn{G,\PP,\QQ}$, given a subset of vertices $W$, we define the graph 
%$\hat G_W$ $G[W]_\QQ$ $G(W,\QQ)$ 
$G[W]/\QQ$ with a vertex $w_j$ for each $j \in [m]$ such that $W \cap Q_j \neq \emptyset$.
Two vertices $w_{j_1}$ and $w_{j_2}$ are adjacent in $G[W]/\QQ$ if and only if there exist $v_1 \in W \cap Q_{j_1}$ and $v_2 \in W \cap Q_{j_2}$ such that $v_1v_2 \in E$.

\begin{proposition}\label{prop.2.step.heuristic}
For any feasible instance $(G, \PP, \QQ)$ of DPCP, $$\dpcn{G,\PP,\QQ}=\min \{\chi(G[W]/\QQ): W \text{ is a feasible set of representatives} \}.$$ 
\end{proposition} 

\begin{proof}
It is enough to prove that the minimum number of colors in a double-partition coloring defined on a feasible set of representatives $W$ is $\cn{G[W]/\QQ}$. 

A double-partition coloring of $(G, \PP, \QQ)$ defined on $W$ corresponds to a total proper coloring of $G[W]/\QQ$ in which each vertex $w_j \in V(G[W]/\QQ)$ is assigned the unique color used in $Q_j$. 
Conversely, since $W\cap Q$ is a stable set of $G$ for every $Q\in \QQ$, a total proper coloring of $G[W]/\QQ$ corresponds to a double-partition coloring of $(G, \PP, \QQ)$ in which each $v \in W$ is assigned the color of $w_{j(v)}$. 
Both constructions preserve the number of colors used.
\end{proof}

Remark \ref{lemma.feasible.solutions} shows that the feasibility of an instance does not depend on the edges between distinct parts of $\QQ$. 
Therefore, by removing these edges while retaining those that make each part of $\PP$ a clique, we can reduce the feasibility problem to the Maximum Stable Set Problem. Formally,

\begin{proposition}\label{Gtilde} Given  $(G,\PP,\QQ)$ with $G=(V,E)$, let $\tilde{G}=(V,\tilde{E})$, with 
$\tilde{E} = \{uv \in E: i(u) = i(v) \vee j(u) = j(v)\}$.
Then, $(G,\PP,\QQ)$ is feasible if and only if $\alpha(\tilde{G})= n$.
\end{proposition}

\begin{proof}
It is clear that $\alpha(\tilde G)\leq n$, as any set of more than $n$ vertices contains at least two vertices belonging to the same $P\in \PP$, which are adjacent in $\tilde{G}$.

Let $(G, \PP, \QQ)$ be a feasible instance of DPCP. By the previous lemma, there exists a set of representatives $W$ such that, for all $Q \in \QQ$, $W \cap Q$ is a stable set of $G$. Then, for any $u, v \in W$ we have $i(u)\neq i(v) \text{ and } j(u)\neq j(v)$, i.e., $uv \notin \tilde{E}$. Therefore, $W$ is a stable set of $\tilde{G}$ and $\card{W} = n$. We can conclude that $\alpha(\tilde{G}) = n$.

Conversely, assume $\alpha(\tilde{G}) = n$ and let $S \subseteq V$ be a stable set of maximum cardinality in $\tilde{G}$. Thus, we have $\card{S \cap P} = 1$ for all $P \in \PP$.
Furthermore, consider two nodes  $u, v \in S \cap Q$, for some $Q\in \QQ$. Since $S$ is a stable set in $\tilde{G}$, $uv \notin \tilde{E}$ and it follows that $uv \notin E$. Therefore, $S \cap Q$ is a stable set of $G$ for each $Q \in \QQ$.
\end{proof}

An example of the construction described in Proposition \ref{Gtilde} is illustrated in Figure \ref{fig:ex-feasibility}. 

\begin{figure}
\centering
\begin{subfigure}{0.41\textwidth}
\centering\begin{tikzpicture}

\draw[draw=white] (0,-1) rectangle (1.5,3.5);

\draw[draw=black!20, fill=black!10, opacity=0.8, line width=1pt] (0.75,3.25) to[closed,curve through={(0,2.5) .. (-0.5,1.75) .. (0.75,1.75) .. (2,1.75) .. (1.5,2.5)}] (0.75,3.25);
\draw[draw=black!20, fill=black!15, opacity=0.8, line width=1pt] (-0.4,1.25) to[closed,curve through={(0.55,0.3) .. (1.85,-0.35) .. (0.95,0.7)}] (-0.4,1.25);
\draw[draw=black!20, fill=black!20, opacity=0.8, line width=1pt] (1.9,1.25) to[closed,curve through={(0.55,0.7) .. (-0.35,-0.35) .. (0.95,0.3)}] (1.9,1.25);

\draw[line width=1pt] (0,0) to (0,1);
\draw[line width=1pt] (0,0) to (1.5,1);
%\draw[line width=1pt] (0,0) to[bend left=90, looseness=0.5] (0,2);

\draw[line width=1pt] (1.5,0) to (1.5,1);
\draw[line width=1pt] (1.5,0) to (0,1);
%\draw[line width=1pt] (1.5,0) to[bend right=90, looseness=0.5] (1.5,2);

\draw[line width=1pt] (0,1) to (0,2);
\draw[line width=1pt] (1.5,0) to (0,2);
\draw[line width=1pt] (0,0) to (1.5,2);
\draw[line width=1pt] (1.5,1) to (1.5,2);

\draw[line width=1pt] (0,2) to (1.5,2);
\draw[line width=1pt] (0,2) to (0.75,3);

\draw[line width=1pt] (1.5,2) to (0.75,3);

\draw[fill=white, draw=black, line width=1pt] (0,0) circle (3.5pt);
%\draw[draw=black, line width=1pt] (0,0) circle (0.75pt);
\draw[fill=white, draw=black, line width=1pt] (1.5,0) circle (3.5pt);
\draw[fill=white, draw=black, line width=1pt] (0,1) circle (3.5pt);
\draw[fill=white, draw=black, line width=1pt] (1.5,1) circle (3.5pt);
\draw[fill=white, draw=black, line width=1pt] (0,2) circle (3.5pt);
\draw[fill=white, draw=black, line width=1pt] (1.5,2) circle (3.5pt);
\draw[fill=white, draw=black, line width=1pt] (0.75,3) circle (3.5pt);

\node at (0.75,2.35) {$P_1$};
\node at (-0.1,-0.35) {$P_2$};
\node at (1.6,-0.35) {$P_3$};

\end{tikzpicture}
\begin{tikzpicture}

\draw[draw=white] (0,-1) rectangle (1.5,3.5);
\draw[draw=black!20, fill=black!10, opacity=0.8, line width=1pt] (-0.25,2.25) to[closed,curve through={(0.5,1) .. (1.8,-0.3) .. (0.9,1)}] (-0.25,2.25);
\draw[draw=black!20, fill=black!14, opacity=0.8, line width=1pt] (-0.6,1) to[closed,curve through={(0.75,1.2) .. (2.2,1) .. (0.75,0.8)}] (-0.6,1);
\draw[draw=black!20, fill=black!18, opacity=0.8, line width=1pt] (1.75,2.25) to[closed,curve through={(0.5,1) .. (-0.3,-0.3) .. (0.9,1)}] (1.75,2.25);
\draw[draw=black!20, fill=black!22, opacity=0.8, line width=1pt] (0.75,3) ellipse (20pt and 10pt) node [left=0.1cm] {$Q_1$};

\draw[line width=1pt] (0,0) to (0,1);
\draw[line width=1pt] (0,0) to (1.5,1);
%\draw[line width=1pt] (0,0) to[bend left=90, looseness=0.5] (0,2);

\draw[line width=1pt] (1.5,0) to (1.5,1);
\draw[line width=1pt] (1.5,0) to (0,1);
%\draw[line width=1pt] (1.5,0) to[bend right=90, looseness=0.5] (1.5,2);

\draw[line width=1pt] (0,1) to (0,2);
\draw[line width=1pt] (1.5,0) to (0,2);
\draw[line width=1pt] (0,0) to (1.5,2);
\draw[line width=1pt] (1.5,1) to (1.5,2);

\draw[line width=1pt] (0,2) to (1.5,2);
\draw[line width=1pt] (0,2) to (0.75,3);

\draw[line width=1pt] (1.5,2) to (0.75,3);

\draw[fill=white, draw=black, line width=1pt] (0,0) circle (3.5pt);
%\draw[draw=black, line width=1pt] (0,0) circle (0.75pt);
\draw[fill=white, draw=black, line width=1pt] (1.5,0) circle (3.5pt);
\draw[fill=white, draw=black, line width=1pt] (0,1) circle (3.5pt);
\draw[fill=white, draw=black, line width=1pt] (1.5,1) circle (3.5pt);
\draw[fill=white, draw=black, line width=1pt] (0,2) circle (3.5pt);
\draw[fill=white, draw=black, line width=1pt] (1.5,2) circle (3.5pt);
\draw[fill=white, draw=black, line width=1pt] (0.75,3) circle (3.5pt);

\node at (0,-0.35) {$Q_2$};
\node at (2,1) {$Q_4$};
\node at (1.5,-0.35) {$Q_3$};

\end{tikzpicture}
\caption{$(G, \PP,\QQ)$}
\label{fig.dpcp}
\end{subfigure}
\begin{subfigure}{0.41\textwidth}
\centering
\begin{tikzpicture}
\draw[draw=white] (0,-1) rectangle (1.5,3.5);
%\node at (-0.75,2.95) {$\tilde G$};
\draw[black, dotted] (0,0) to (0,1);
\draw[line width=1pt] (0,0) to (1.5,1);
%\draw[line width=1pt] (0,0) to[bend left=90, looseness=0.5] (0,2);

\draw[black, dotted] (1.5,0) to (1.5,1);
\draw[line width=1pt] (1.5,0) to (0,1);
%\draw[line width=1pt] (1.5,0) to[bend right=90, looseness=0.5] (1.5,2);

\draw[black, dotted] (0,1) to (0,2);
\draw[line width=1pt] (1.5,0) to (0,2);
\draw[line width=1pt] (0,0) to (1.5,2);
\draw[black, dotted] (1.5,1) to (1.5,2);

\draw[line width=1pt] (0,2) to (1.5,2);
\draw[line width=1pt] (0,2) to (0.75,3);

\draw[line width=1pt] (1.5,2) to (0.75,3);

\filldraw[fill=white, draw=black, line width=1pt] (0,0) circle (3.5pt); %node [below left] {\footnotesize $(E_3,v_3)$};
%\filldraw[draw=black, line width=1pt] (0,0) circle (0.75pt);
\filldraw[fill=white, draw=black, line width=1pt] (1.5,0) circle (3.5pt); %node [below right] {\footnotesize $(E_2,v_1)$};
\filldraw[fill=white, draw=black, line width=1pt] (0,1) circle (3.5pt); %node [left=0.25cm] {\footnotesize $(E_2,v_2)$};
\filldraw[fill=white, draw=black, line width=1pt] (1.5,1) circle (3.5pt); %node [right=0.25cm] {\footnotesize $(E_3,v_2)$};
\filldraw[fill=white, draw=black, line width=1pt] (0,2) circle (3.5pt); %node [above left] {\footnotesize $(E_1,v_1)$};
\filldraw[fill=white, draw=black, line width=1pt] (1.5,2) circle (3.5pt); %node [above right] {\footnotesize $(E_1,v_3)$};
\filldraw[fill=white, draw=black, line width=1pt] (0.75,3) circle (3.5pt); %node [above] {\footnotesize $(E_1,v_4)$};
\end{tikzpicture}
\caption{$\tilde G=(V,\tilde E)$}
%\label{fig.conflict.graph}
\end{subfigure}
\caption{
Instance $(G,\PP,\QQ)$ is feasible if and only if $\tilde G$ contains a stable set of size $|\PP| = 3$. %{\color{blue} $\tilde G$ se obtiene eliminando de $G$ las aristas que unen vértices que pertenecen a distintas partes de $\PP$ y distintas partes de $\QQ$ }
}
\label{fig:ex-feasibility}
\end{figure}
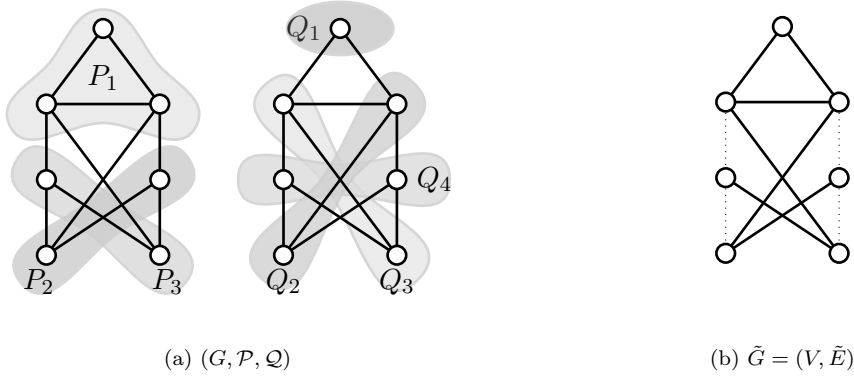

\section{Heuristic approaches for DPCP}
\label{sec:heuristics}

This section describes heuristic approaches for DPCP based on techniques originally developed for PCP. A key difference between the two problems is that DPCP may admit infeasible instances, which makes feasibility a central concern in the design of these heuristics.

For PCP, \citet{li2000partition} propose two strategies depending on whether vertex selection and color assignment are carried out sequentially or simultaneously.
\emph{Two-step} heuristics first select a vertex from each part of the first partition and then color the subgraph induced by these vertices.
In contrast, in \emph{one-step} heuristics, a valid color is immediately assigned to a selected vertex before the next vertex is chosen.
The criteria according to which vertices and colors are selected are the main design choices and give rise to different, yet related, heuristics.
Among these, the color-degree criteria, inspired by the DSatur heuristic for VCP \citep{brelaz1979new}, exhibited the best performance, particularly within the one-step approach, referred to as \emph{onestepCD}.

% Mejores heuristicas de PCP

% \textbf{onestepCD:} [Preprocesamiento se eliminan todas las aristas cuyos vértices estén en el mismo grupo.] Se busca el vértice con el menor grado de saturación para cada grupo sin color. Si hay un empate, se elige el vértice con el menor uncolored degree (número de vértices adyacentes sin color hasta el momento). De estos vértices, se encuentra el vértice con el mayor grado de saturación. Si hay un empate, se elige el vértice con el mayor uncolored degree. Se asigna a este vértice el color más pequeño disponible. Se eliminan todos los demás vértices del mismo grupo. Se repite todo el proceso hasta que todos los grupos estén coloreados. 

% \textbf{twostepCD:} [Preprocesamiento se eliminan todas las aristas cuyos vértices estén en el mismo grupo.] De todos los grupos sin procesar, elige el vértice con el menor grado y eliminan todos los demás vértices del mismo grupo que el vértice seleccionado. Repiten este proceso hasta que cada grupo tenga exactamente un vértice seleccionado. Todos estos vértices seleccionados inducen un subgrafo del grafo original. Una vez construido el subgrafo inducido, colorear con DSatur.

In what follows, we consider a DPCP instance $(G, \PP, \QQ)$, with $G = (V,E)$, $\PP = \{P_i: i \in [n]\}$, and $\QQ = \{Q_j: j \in [m]\}$.
For convenience, this section assumes that each $P_i \in \PP$ is a stable set of $G$, so that the degree of $v$ is not influenced by the size of $P_{i(v)}$.
Any tie arising in the application of the heuristics is ultimately resolved according to an index-based criterion.

\subsection{Two-step heuristics}\label{sec:heuristics:2:step}

For DPCP, the vertex-selection phase aims to construct, whenever possible, a feasible set of representatives $W$. 
Algorithm \ref{algo.2-step.I} presents a greedy heuristic for constructing the set $W$ iteratively.
Let $I$ denote the set of indices corresponding to the uncolored parts of $\PP$.
At each iteration, an uncolored part $P_i$ with $i \in I$ and a vertex $v \in P_i$ are chosen according to a given criterion.
Before proceeding to the next iteration, all vertices in $P_{i} \setminus \{v\}$ and all neighbors of $v$ in $Q_{j(v)}$ are removed from the instance (from the graph and the partitions), thereby avoiding superfluous vertices and ensuring that $W \cap Q_{j(v)}$ remains a stable set, respectively.
After these removals, it may happen that some part of $\PP$ becomes empty, in which case the heuristic fails.

\begin{algorithm}
\caption{greedy heuristic for the vertex-selection phase}
\label{algo.2-step.I}
\small
\KwData{$G=(V,E)$ and partitions $\PP$,$\QQ$ of $V$}
$W \gets \emptyset$; $I \gets [n]$\;
\While{$I \neq \emptyset$}{
select $i \in I$ and $v \in P_i$ with some criterion\;
$W \gets W \cup \{v\}$; $I \gets I \setminus \{i\}$\;
remove vertices in $P_{i} \setminus \{v\}$ and $N(v) \cap Q_{j(v)}$ from the instance\;
\If{$P_{k} = \emptyset$ for some $k \in I$}{
    \textbf{fail}\;
}
}
\textbf{return} $W$\;
\end{algorithm}

Regardless of the vertex-selection criterion, Algorithm \ref{algo.2-step.I} is guaranteed to terminate successfully whenever every $Q_j \in \QQ$ induces a stable set in $G$. 
This immediately yields the following result.

\begin{remark}\label{remark.2.step.heuristic.success}
For instances of DPCP arising from VCP, PCP and CFCP, Algorithm \ref{algo.2-step.I} always succeeds.
\end{remark}

However, for general instances of DPCP, as well as for instances arising from LCP, Algorithm \ref{algo.2-step.I} may fail to produce a feasible solution. 
To increase the likelihood of obtaining a feasible solution, in line 3 we select an uncolored set $P_i \in \PP$ of minimum cardinality. 
We then consider three criteria for choosing the vertex $v$ from $P_i$.
All criteria are evaluated on the current graph at the corresponding iteration. Ties are broken by selecting the vertex $v$ with the fewest neighbors in $Q_{j(v)}$.
\begin{itemize}
\item[] \textsc{Deg.} Select a vertex of minimum degree (used by \cite{li2000partition} for PCP).
\item[] \textsc{QDeg.} Select a vertex $v$ whose neighborhood intersects the fewest parts of $\QQ$. An adaptation of \textsc{Deg} that takes into account the future contraction of the parts of $\QQ$ before coloring.
\item[] \textsc{SQDeg.} 
For each candidate vertex $v$, compute its degree in the graph obtained by contracting the parts of $\QQ$ in $G[W \cup \{v\}]$, and select the vertex with minimum degree. Here, $W$ denotes the set of already selected vertices. Unlike \textsc{Deg} and \textsc{QDeg}, this criterion ignores adjacencies involving non-selected vertices, many of which may never be selected.
\end{itemize}

For the coloring phase, Proposition \ref{prop.2.step.heuristic} implies that the problem can be reduced to solve VCP on the graph $G[W]/\QQ$.  Although any VCP algorithm could be used for this task, an optimal coloring of $G[W]/\QQ$ does not necessarily yield an optimal double-partition coloring of the original DPCP instance, since the set $W$ is fixed. 
Therefore, we employ the classical DSatur heuristic of \cite{brelaz1979new} to color $G[W]/\QQ$. 

Motivated by the good performance of the two-step heuristics, which will be discussed in Section \ref{sec:B&Pconfigurations}, we propose a semi-greedy variant of Algorithm \ref{algo.2-step.I} by introducing randomness through a restricted candidate list (RCL) \citep{Feo1995}.
Vertices whose value under the selection criterion lies within $\alpha\%$ of the best are included in the RCL, and one of them is chosen at random and added to $W$.
The modified heuristic is run a fixed number of times, each time followed by the coloring phase. 
The best solution found, i.e., the one using the fewest colors, is returned.

\subsection{One-step heuristics}\label{sec:heuristics:1:step}

The overall structure of the one-step greedy heuristics for DPCP is fairly similar to Algorithm \ref{algo.2-step.I}, with two additional features:
a color is immediately assigned once a vertex is selected, and additional vertices that conflict with this assignment are removed. 
The resulting coloring is returned at the end.
Further details on the required modifications are specified below.

After a vertex $v$ is selected in line 3, its color must be carefully chosen to maintain feasibility, unlike in PCP, where any color not used by its neighbors can be assigned.
This implies that, if $v$ is not the first colored vertex in $Q_{j(v)}$, it must receive the color already used in $Q_{j(v)}$.
Otherwise, a color not used by the neighbors of $v$ must be selected.
However, this choice of color is not arbitrary, since each color invalidates a specific set of vertices.
Specifically, assigning a color $c$ to $v$ removes all vertices in $Q_{j(v)}$ that are adjacent to a vertex colored with $c$, as well as any neighbor $u$ of $v$ such that $j_u \neq j_v$ and $Q_{j(u)}$ already contains a vertex colored with $c$.
These vertices are removed after line 5, following the removals triggered by the selection of $v$.

Our vertex-selection criterion is directly inspired by \emph{onestepCD} for PCP, but we adopt a different tie-breaking strategy.
First, for each uncolored part $P_i \in \PP$, a candidate vertex is selected. The candidate is a vertex of minimum saturation degree in $P_i$, where the saturation degree corresponds to the number of distinct colors appearing in its neighborhood. 
Ties are broken analogously to the two-step heuristics, considering only uncolored neighbors. 
Finally, among all candidates, we select a vertex $v$ of maximum saturation degree, breaking ties in favor of the candidate whose part $P_{i(v)}$ has minimum cardinality.
If the color of $v$ is not forced, we assign a color not used by its neighbors that invalidates the minimum number of vertices, avoiding colors that would make any part in $\PP$ empty.
If no such color exists, a new color is assigned to $v$, increasing the chromatic number.

%From a theoretical perspective, one-step heuristics are more failure-prone than two-step heuristics, as both vertex and color selection may trigger vertex removals. Indeed, when the color of a selected vertex is forced by previous decisions, failure may be unavoidable in the absence of repair mechanisms.
%Consequently, under this constructive setting for one-step heuristics \textcolor{red}{(con esto quiero decir, sin reparar o volver atras en nuestras decisiones, ¿se entiende?)}, the existence of competitive vertex- and color-selection criteria that guarantee success on instances of DPCP arising from VCP, PCP, and CFCP remains an open question. \todo{Eliminar esta ultima oracion? }\todo{De acuerdo!}

\section{A Branch-and-Price algorithm for solving DPCP}
\label{sec:branch-and-price-algo}

There are two predominant approaches in the literature for formulating graph coloring problems as integer linear programs (ILPs). The first consists of compact formulations that use decision variables to assign colors to vertices. The second relies on so-called \emph{set covering formulations}, which involve an exponential number of binary decision variables, each corresponding to a stable set of the graph and representing a potential color class. The following formulation for the DPCP, proposed in \cite{LucciEtAl2025}, belongs to the latter class.

Given an instance $(G,\PP,\QQ)$, let $\STAB(G)$ denote the collection of all stable sets of $G$. For each $S \in \STAB(G)$, we define a binary variable $x_S$ that takes value $1$ if and only if the stable set $S$ is \emph{active}, i.e., if it is selected as a color class in the solution. Then, the DPCP can be formulated as follows:
\begin{align}
(\setcoveringform) \quad \min & \sum_{S \in \STAB(G)} x_S \label{SCFormulation.Obj}\\
s.t. & \sum_{S \in \STAB(G): P \cap S \neq \emptyset} x_S \geq 1, &  P \in \PP, \label{SCFormulationA} \\
    & \sum_{S \in \STAB(G): Q \cap S \neq \emptyset} x_S \leq 1, &  Q \in \QQ, 
    %\; \color{red} \card{Q} \geq 2,\color{black} 
    \label{SCFormulationB} \\
& x \in \{0,1\}^{\STAB(G)}. \nonumber
\end{align}

The objective function~\eqref{SCFormulation.Obj} minimizes the number of active stable sets (color classes) used in the solution. Constraints~\eqref{SCFormulationA} ensure that, for each $P \in \PP$, at least one vertex of $P$ belongs to an active stable set and is therefore colored. Finally, constraints~\eqref{SCFormulationB} ensure that each set $Q \in \QQ$ intersects at most one active stable set, i.e., all colored vertices in $Q$ receive the same color.
%\todo{Está un poco enredado. Yo diria que cada $Q$ se interecta con a lo sumo un estable activo}

Observe that, for any singleton $\{v\} \in \QQ$, the corresponding constraint in~\eqref{SCFormulationB} can be removed from \setcoveringform\ without affecting the optimal value. Indeed, removing such a constraint may yield a solution in which vertex $v$ belongs to multiple color classes. In that case, removing $v$ from all but one of these color classes produces a feasible solution to \setcoveringform\ with the same objective value.

Hence, throughout the remainder of the paper we assume that constraints~\eqref{SCFormulationB} are enforced only for sets $Q \in \QQ$ with $|Q| \geq 2$, i.e., they are replaced by
\begin{equation}
    \sum_{S \in \STAB(G): Q \cap S \neq \emptyset} x_S \leq 1, \quad  Q \in \QQ, \; \card{Q} \geq 2.
    \tag{\ref{SCFormulationB}'}
    \label{SCFormulationB2}
\end{equation}

In the remainder of this section, we describe the main components of a B\&P algorithm for solving \setcoveringform. The pricing problem, the initialization procedure, and the branching scheme are presented in the following subsections. Throughout this section, we assume that an instance $(G,\PP,\QQ)$ is given, with $G=(V,E)$, $\PP=\{P_i:i\in[n]\}$, and $\QQ=\{Q_j:j\in[m]\}$. We also assume that each set $P_i$ induces a clique in $G$, since all missing edges are added during a preprocessing step.

For a comprehensive overview of the classical components of a B\&P framework, we refer the reader to \cite{Desrosiers2026} and, in the context of graph coloring problems, to \cite{Trick1996,Furini2018}.

\subsection{Pricing problem}\label{sec:branch-and-price-algo-pricing}

In the following, we assume without loss of generality that constraints~\eqref{SCFormulationB} (or, equivalently, constraints~\eqref{SCFormulationB2}) in formulation \setcoveringform\ have been multiplied by $-1$. Therefore, given a stable set $S \in \STAB(G)$, its associated column in \setcoveringform\ is the $0/1$ vector $x = (x^P, -x^Q) \in \{0,1\}^{n+m}$, where $x^P$ and $x^Q$ are the characteristic vectors of the sets
$\PP(S) = \{i(v) \in [n]: v \in S\}$ and $\QQ(S) = \{j(v) \in [m]: v \in S\}$,
respectively.

Observe that two distinct stable sets $S, S' \in \STAB(G)$ may correspond to the same column if $\PP(S) = \PP(S')$ and $\QQ(S) = \QQ(S')$. In our implementation, each column is linked to a single stable set, which allows colorings to be reconstructed from solutions of \setcoveringform.
Given a subset $\hat{\STAB} \subseteq \STAB(G)$, denote by $\setcoveringform(\hat{\STAB})$ the restriction of \setcoveringform\ to the variables associated with stable sets in $\hat{\STAB}$. The pricing problem consists of determining, given an optimal solution $x^*$ to the linear relaxation of $\setcoveringform(\hat{\STAB})$, whether there exists a column associated to a stable set in $\STAB(G) \setminus \hat{\STAB}$ with negative reduced cost.

%{\color{red}pondria parrafos 2 y 3 como parrafos 1 y 2, y el 1 como tercero para enganchar mejor}

Let $(\gamma^*, \mu^*) \in \mathbb{R}_+^{n+m}$ be an optimal solution to the dual of the linear relaxation of $\setcoveringform(\hat{\STAB})$. For any stable set $S \in \STAB(G)$, the reduced cost of the corresponding column is given by
\[
\bar{c}_S = 1 - \left( \sum_{i \in \PP(S)} \gamma^*_i - \sum_{j \in \QQ(S)} \mu^*_j \right).
\]
Hence, the pricing problem consists of finding a stable set $S \in \STAB(G) \setminus \hat{\STAB}$ such that $\bar{c}_S < 0$, or proving that no such set exists. Note that this condition is equivalent to
\[
\sum_{i \in \PP(S)} \gamma^*_i - \sum_{j \in \QQ(S)} \mu^*_j > 1.
\]

\begin{proposition}
The pricing problem is NP-complete.
\end{proposition}

\begin{proof}
It is immediate that the pricing problem belongs to NP. We prove NP-hardness by reduction from the decision version of the Weighted Stable Set Problem (WSSP).

An instance $(G,w,k)$ of the WSSP consists of a graph $G=(V,E)$, a vector $w \in \mathbb{Z}_+^V$ of nonnegative integer vertex weights, and an integer $k \in \mathbb{Z}_+$. The problem asks whether $G$ contains a stable set of weight at least $k$.

Consider the pricing problem defined on the same graph $G$, with partitions $\PP = \{\{v\} : v \in V\}$ and $\QQ = \{V\}$, and dual variables defined by $\gamma_v = w_v/(k-1)$ for all $v \in V$, and $\mu_V = 0$.
Then, for any stable set $S \in \STAB(G)$, we have
$\bar{c}_S = 1 - \frac{w(S)}{k-1}$.
Since $w(S)$ and $k$ are integers, there exists a stable set with $w(S) \ge k$ if and only if there exists a stable set with $\bar{c}_S < 0$. This completes the reduction.
\end{proof}

%{\color{red} no pondria los parrafos en negro, solo algo mejorado de lo que esta en rojo, sin lo de formulacion edge, diria que usamos una formulacion por cubrimiento de arcos por cliques}
%The proof above shows that any instance of the Weighted Stable Set Problem can be formulated as an instance of the pricing problem. Conversely, if $\PP = \QQ = \{\{v\} : v \in V\}$, then any instance of the corresponding pricing problem (for any rational choice of dual variables) can be formulated as an instance of the Weighted Stable Set Problem. In general, however, the pricing problem cannot be reduced to the Weighted Stable Set Problem.

\begin{comment}
Given $\hat \STAB \subseteq \STAB(G)$, let $x^*$ be an optimal solution for the linear relaxation of $\setcoveringform(\hat \STAB)$ and $(\gamma^*,\mu^*) \in \R_+^{A \times B}$ be its corresponding dual optimal solution.
\end{comment}
Using binary variables $s_v$, for $v \in V$, such that $s_v = 1$ if and only if $v \in S$, and binary variables $w_j$, for $j \in [m]$, such that $w_j = 1$ if and only if $Q_j \cap S \neq \emptyset$, the pricing problem can be formulated as the following integer linear program:
\begin{comment}
    The pricing problem asks to find (if any) a stable set $S^* \in \STAB(G) \setminus \hat \STAB$ such that $\sum_{a \in A(S^*)} \gamma^*_a - \sum_{b \in B(S^*)} \mu^*_b > 1$.
It can be addressed by the following ILP formulation. 
For each $(a,b) \in V$, there is a binary decision variable $s_{ab}$ such that $s_{ab} = 1$ if and only if $(a,b)$ belongs to $S^*$. 
For each $b \in B$, there is a binary decision variable $w_b$ such that $w_b = 1$ if $b \in B(S^*)$.
\end{comment}
%\textcolor{red}{Tenemos que resolver una variante del problema de estable de peso maximo, donde el peso del estable no es la suma de los pesos de los vertices, sino que usamos tal funcion objetivo. La formulacion con restricciones EDGE se puede reforzar, teniendo en cuenta la estrucutura de nuestras instancias, de la siguiente forma.}
{\small
\begin{align}
(\pricingform) \quad \max & \sum_{i \in [n]} \gamma^*_i \sum_{v \in P_i} s_v - \sum_{j \in [m]} \mu^*_j w_j \label{PricingFormulation.Obj}\\[-4pt]
\textrm{s.t.} \nonumber \\[-4pt]
& \sum_{v \in P_i} s_v \leq 1, &  i \in [n], \label{PricingFormulation.A.clique} \\
& s_v + \sum_{\substack{u \in N(v) \cap P_k }} s_u \leq 1, &  v \in V,\ k \in [n] \setminus \{i(v)\}, \label{PricingFormulation.A.edge} \\
& s_v \leq w_{j(v)}, &  v \in V, \label{PricingFormulation.B}  \\
& s \in \{0,1\}^V,\ w \in \{0,1\}^{m}. \nonumber
\end{align}
}

The objective function~\eqref{PricingFormulation.Obj} ensures that an optimal solution $S^*$ maximizes the dual contribution and, equivalently, minimizes the reduced cost $\bar c_{S^*} = 1 - (\sum_{i \in \PP(S^*)} \gamma^*_i - \sum_{j \in \QQ(S^*)} \mu^*_j)$.

Constraints~\eqref{PricingFormulation.A.clique} and~\eqref{PricingFormulation.A.edge} correspond to clique inequalities associated with the clique $P_i$, for $i \in [n]$, and the clique $\{v\} \cup (N(v) \cap P_k)$, for $v \in V$ and $k \in [n] \setminus \{i(v)\}$, respectively. 
Observe that each edge of $G$ is covered by at least one of these cliques, which enforces that $S^*$ is a stable set.
Constraints~\eqref{PricingFormulation.B} ensure that $w_j$ is activated whenever $S^*$ contains at least one vertex in $Q_j$.

%These inequalities are strengthened by requiring activation whenever a vertex in the clique $\{v\} \cup (N(v) \cap P_k \cap Q_{j(v)})$ is selected, {\color{red}for $v \in V$ and $k \in [n] \setminus \{i(v)\}$}.

\begin{comment}
\textcolor{red}{Como es usual, no necesitamos una optima, sino cualquiera con costo reducido negativo.}
An early stop can be introduced when optimizing \pricingform, as soon as a feasible solution with an objective value greater than 1 is found.
\textcolor{red}{En realidad paramos cuando encontramos un estable de peso $>$ 1.1, ver si vale la pena decirlo.}
Even with this improvement, optimization can be too time-consuming.
For this reason, SE PROPONEN HEURISTICAS.
\textcolor{red}{esto ya se dice en el parrafo de abajo. elegir uno de los dos}
\end{comment}

\subsubsection{Pricing heuristics}

An instance of the pricing problem must be solved at every iteration of the column generation process, at each node of the branching tree. To reduce the associated computational effort, we employ three pricing algorithms. Only if none of them succeeds in finding a stable set with negative reduced cost, formulation DPCP-PRI is solved using a commercial ILP solver.

\paragraph{Heuristic pricing}

We propose a greedy heuristic that constructs a stable set by repeatedly selecting a vertex of maximum weight.
Unlike classical heuristics for the WSSP, where vertex weights are fixed, the weights in our pricing problem are adaptive, as they depend on previously selected vertices. 
Initially, each vertex $v$ is assigned weight $\gamma^*_{i(v)} - \mu^*_{j(v)}$. Upon selecting a vertex $v$, all its neighbors are removed, and the weight of each remaining vertex $u \in Q_{j(v)}$ is updated to $\gamma^*_{i(u)}$. %If the total weight of the resulting stable set exceeds $1$, a candidate column is generated.

This heuristic is further enhanced by a semi-greedy strategy based on a restricted candidate list \citep{Feo1995}, similar to the two-step heuristic previously proposed for the DPCP in Section \ref{sec:heuristics:2:step}. The resulting randomized variant is executed a fixed number of times, potentially generating multiple columns.

\paragraph{WSSP-based approach}

By ignoring the adaptive behavior of the vertex weights and using fixed weights instead, the pricing problem can be approximated by the WSSP. After a stable set is obtained, its value is recomputed using the original pricing objective function.

Given a stable set $S \in \STAB(G)$, let $w(S)$ denote its weight under the objective function~\eqref{PricingFormulation.Obj}. Two fixed-weight variants are considered:
\[
w_{\PP\QQ}(S) = \sum_{v \in S} \left(\gamma^*_{i(v)} - \mu^*_{j(v)}\right)
\quad \text{and} \quad
w_{\PP}(S) = \sum_{v \in S} \gamma^*_{i(v)}.
\]
Observe that
$$w_{\PP\QQ}(S) \leq w(S) \leq w_{\PP}(S).$$

In the first approach, each vertex $v$ is assigned weight $\gamma^*_{i(v)} - \mu^*_{j(v)}$, and vertices with negative weight are discarded. For any stable set $S$ such that $w_{\PP\QQ}(S) > 1$, it follows that $w(S) > 1$, and consequently a candidate column is generated. However, even when no such stable set exists, a stable set $S^*$ maximizing $w_{\PP\QQ}(S^*)$ may still be useful if it satisfies $w(S^*) > 1$.

Additionally, before evaluating the weight in the original pricing problem, the stable set can be greedily extended by adding vertices from any set $Q \in \QQ$ that already contains selected vertices, since they have nonnegative contribution to the objective function. 

In the second approach, each vertex $v$ is assigned weight $\gamma^*_{i(v)}$, and a stable set $S^*$ maximizing $w_{\PP}(S^*)$ is computed. If $w_{\PP}(S^*) \leq 1$, then $w(S) \leq 1$ for all stable sets $S \in \STAB(G)$, implying that no column with negative reduced cost exists and the column generation process can be safely terminated without solving the ILP formulation. Otherwise, the weight of $S^*$ is evaluated in the original problem, and if $w(S^*) > 1$, a candidate column is obtained.

\subsection{Initialization}\label{sec:branch-and-price-algo-init}

To start the column generation process, an initial set of columns $\hat \STAB \subseteq \STAB(G)$ is required such that the linear relaxation of $\setcoveringform(\hat \STAB)$ is feasible.
For feasible instances of \newproblem, the set of active columns in any feasible solution provides a valid choice for $\hat \STAB$.
However, as discussed in Section~\ref{sec:feasibility:DPCP}, deciding feasibility of a general instance of \newproblem\ is NP-complete.
Consequently, initializing general instances of \newproblem\ is more challenging than for instances arising from the VCP, PCP, or CFCP, for which feasible solutions can be easily constructed.

The initialization procedure consists of the following steps.
First, we attempt to compute a heuristic solution using the approaches described in Section~\ref{sec:heuristics}.
%If this fails, we attempt to decide feasibility via the reduction to the Maximum Stable Set Problem described in Proposition~\ref{Gtilde}.
%Although this problem is NP-hard as well, highly optimized algorithms and software are available that can solve many large instances to optimality within seconds.
%If a maximum stable set can be computed within a prescribed time limit, Proposition~\ref{Gtilde} is applied to certify either feasibility or infeasibility of the instance; in the former case, the stable set is then used to define a feasible solution.
%Otherwise, 
If the heuristics fail, 
%and the time limit for the Maximum Stable Set Problem is insufficient to decide feasibility, 
we resort to computing a basic feasible solution of $\setcoveringformLR$, the linear relaxation of \setcoveringform.
A common approach for initializing linear programs consists of solving auxiliary problems with artificial variables.
%We propose solving the following linear program, which includes a single dummy variable $t$ with a suitably chosen coefficient $M \in \Z_+$ in the objective function.
To this end, we solve the following parameterized linear program, which includes an additional variable $t$ with objective coefficient given by a parameter $M\in\mathbb Z_+$:
\begin{align}
(\setcoveringformLRM) \quad  &   \min \sum_{S \in \STAB(G)} x_S+ M t
\color{black}
\\
 & \sum_{S \in \STAB(G): P \cap S \neq \emptyset} x_S + t \geq 1, &  P \in \PP, \label{SCFormulation.A} \\
& \sum_{S \in \STAB(G): Q \cap S \neq \emptyset} x_S \leq 1, &  Q \in \QQ, \label{SCFormulation.B} \\
& x \in [0,1]^{\STAB(G)} \, , \, t \in [0,1]. \nonumber
\end{align}

%Based on the optimal value of $\setcoveringformLRM$ and the value of $t$ in an optimal solution, the following results hold for \setcoveringformLR.
The following remark provides sufficient conditions to infer feasibility of $\setcoveringformLR$:

\begin{remark}
\label{remark.M}
For any $M \in \Z_+$, let $(x^*, t^*)$ be an optimal solution of $\setcoveringformLRM$ with optimal value $z^*_{\mathrm{ext}}$.
\begin{enumerate}
    \item If $t^* = 0$, then $x^*$ is an optimal solution of $\setcoveringformLR$.
    
    \item Since $\min\{n,m\}$ is an upper bound on the optimal value of $\setcoveringformLR$, it follows that if $t^* > 0$ and $z^*_{\mathrm{ext}} > \min\{n,m\}$, then $\setcoveringformLR$ is infeasible.
\end{enumerate}
\end{remark}

The choice of parameter $M$ plays an important role in practice. The arguments used in the proof of Lemma~4.1 in \cite{Lucci2024} can be adapted to show that there exists a value $M_0 \in \mathbb{R}$ such that, for every $M \ge M_0$, an optimal solution of $\setcoveringformLRM$ satisfies one of the two conditions stated in Remark~\ref{remark.M}. Since excessively large values of $M$ may lead to numerical instabilities, in our implementation we adopt the same value proposed in that work.
%\textcolor{red}{con un raxzonamiento similar al usado en teorema tal de tal paper, se puede probar que Existe teoricamente un valor suficientemente grande de $M$ que garantiza que ...}
%Larger values of $M$ increase the likelihood that solving $\setcoveringformLRM$ yields one of the two outcomes described in Remark~\ref{remark.M}. 
%{\color{red} Por lo tanto elegimos M=1000 como en el paper de coloreo.}

If neither condition of Remark~\ref{remark.M} applies, that is, if $t^* > 0$ and $z^*_{\mathrm{ext}} \leq \min\{n,m\}$, we proceed with a standard two-phase initialization scheme.
In the first phase, we solve the linear program $\setcoveringformLRext$ 
obtained from $\setcoveringformLRM$ by replacing the objective function with the minimization of $t$.
Observe that $\setcoveringformLR$ is feasible if and only if $\setcoveringformLRext$ has optimal value zero.
In that case, the optimal solution of $\setcoveringformLRext$ is used to initialize $\setcoveringformLR$, which is solved in the second phase.

\subsection{Branching scheme}\label{subsec:branching}

The branching scheme determines how the solution space is partitioned whenever a fractional solution is obtained at a node of the B\&P tree. 
In the following, 
we extend to the DPCP a branching rule proposed by \cite{Furini2018} for the PCP. 
%This rule guarantees that every child node of the enumeration tree corresponds to another instance of the DPCP. Moreover, for instances of the PCP (resp.~CFCP) formulated as instances of the DPCP, the child nodes are also instances of the same problem.

%{\color{red} me parece que lo que hay que decir es que la extension aplicada a instancias del PCP coincide con la de Furini. Extendemos la regla esa, de manera que sobre instancias del DPCP correspondientes a PCP coincide con ella. Algo asi. Y nada mas. Que los hijos en general son instancias del DPCP es lo que se espera de una regla de branching en un B\&P}

%Moreover, when applied to PCP instances, this branching scheme coincides with that proposed in \cite{Furini2018}.

%Let 
%\[
%\PP^{\geq 2}=\{P_i\in\PP:%\card{P_i}\geq2\}
%\]
%be the family containing all parts of $\PP$ with cardinality larger than or equal to two.

First, observe that, if $\card{P}=1$ holds for all $P\in\PP$, then Proposition~\ref{prop.2.step.heuristic} implies that the DPCP instance can be transformed into a VCP instance. 
%The same situation arises in the branching scheme for the PCP. 
For PCP instances satisfying this condition, \cite{Furini2018} 
%the authors 
switch to the classical VCP branching scheme and continue the B\&P procedure. Instead, we solve the current node using the B\&P algorithm proposed in \cite{Held2012}. Furthermore, we use the current upper bound (i.e., the value of the best solution found so far) to terminate this secondary B\&P procedure as soon as its lower bound reaches or exceeds this value. 

Our branching rule is applied only if there is at least one part $P\in \PP$ with $\card{P} \geq 2$. It consists in 
selecting such a part $P$ with $|P| \geq 2$ and containing the largest number of vertices that appear in the fractional solution $x^*$, that is, maximizing 
%$\left|\left\{v\in P:
%\sum_{S\in\hat{\STAB}:v\in S}x^*_S>0
%\right\}\right|$ 
$
\left|P\cap (\cup_{S\in\hat{\STAB}: x^*_S>0}S)\right|
$. 
Ties are broken by preferring smaller parts, and then by the index of the smallest part.
After selecting a part $P$, we choose a vertex $v\in P$ maximizing its fractional assignment, i.e., $\sum_{S\in\hat{\STAB}:v\in S}x^*_S$.
The solution space is then partitioned into two subspaces: those double-partition colorings in which $v$ is colored and those in which $v$ remains uncolored.
Both cases lead to new DPCP instances, as described below.
\begin{enumerate}
\item If $v$ is colored, then all vertices in $R=P_{i(v)}\setminus\{v\}$ need no longer be considered, since the requirement of coloring the part $P_{i(v)}$ is already satisfied.
Therefore, they can be removed.
This yields an instance $(G^1,\PP^1,\QQ^1)$ where $G^1=G\setminus R$, $\PP^1=\{P^1_i:i\in[n]\}$ with $P^1_i=P_i$ for all $i\neq i(v)$ and $P^1_{i(v)}=\{v\}$, and $\QQ^1=\{Q_j\setminus R:j\in[m]\}$.
Observe that $P^1_{i(v)}$ is now a singleton, and Lemma~\ref{th:singletons} can be applied.

\item If $v$ is uncolored, then it can be removed, and another vertex in $P_{i(v)}$ must be colored instead.
This leads to an instance $(G^2,\PP^2,\QQ^2)$ where $G^2=G\setminus\{v\}$, $\PP^2=\{P^2_i:i\in[n]\}$ with $P^2_i=P_i$ for all $i\neq i(v)$ and $P^2_{i(v)}=P_{i(v)}\setminus\{v\}$, and $\QQ^2=\{Q^2_j:j\in[m]\}$ with $Q^2_j=Q_j$ for all $j\neq j(v)$ and $Q^2_{j(v)}=Q_{j(v)}\setminus\{v\}$.
\end{enumerate}

% The branching rule proposed by \cite{Trick1996} for 
% the VCP consists of selecting a pair of nonadjacent vertices $u$ and $v$, and partitioning the solution space into those colorings in which $u$ and $v$ receive different colors and those in which they are assigned the same color. Thus, every child node in their branch-and-price tree corresponds to another VCP instance. Since every instance of the VCP is feasible, all nodes of the enumeration tree are feasible.
%  Similarly, the branching scheme proposed by \cite{Furini2018} creates child nodes corresponding to PCP instances, which are therefore always feasible.

Since the cardinality of the selected part $P$ decreases in both child nodes, after applying the branching rule a finite number of times, the case $\card{P}=1$, for all $P\in\PP$ is eventually reached at the leaves of the enumeration tree.

For DPCP instances arising from PCP and CFCP,  each child node again corresponds to an instance of the PCP or the CFCP, respectively, which are known to be always feasible.
In contrast, this branching scheme may generate infeasible child nodes for general DPCP instances, even when the original instance is feasible, as illustrated in the following example.
\begin{example}\label{example:infeasible:branch}
Let $G=P_3$ be the path graph with vertices $v_1,v_2,v_3$, and let 
$\PP=\{\{v_1\},\{v_2,v_3\}\}$ and 
$\QQ=\{\{v_1,v_2\},\{v_3\}\}$.
Clearly, this instance is feasible, since $v_1$ and $v_3$ can be assigned different colors.
If we branch on $v_2$, then the resulting instance $(G^1,\PP^1,\QQ^1)$ is infeasible.
Similarly, branching on $v_3$ leads to an infeasible instance $(G^2,\PP^2,\QQ^2)$.
In both cases, vertices $v_1$ and $v_2$ are forced to receive the same color despite being adjacent, leading to a contradiction.
\end{example}

%The branching scheme presented above 
%is not complete, as it
%cannot be applied directly to nodes of the branch-and-price tree for which $\PP^{\geq 2}= \emptyset$, i.e., to DPCP instances in which $|P_i|=1$ holds for all $i\in[n]$.
%For such instances, a fractional optimal solution indicates which vertices are colored, but a vertex may be fractionally assigned to multiple colors. A similar situation arises in PCP, which is why \cite{Furini2018} switch to the classical VCP branching scheme and continue the branch-and-price procedure.
%Instead, we apply Proposition~\ref{prop.2.step.heuristic} to transform these instances into VCP instances and solve the current node using the branch-and-price algorithm proposed in \cite{Held2012}.
%Furthermore, we use the current upper bound (i.e., the value of the best solution found so far) to terminate this secondary branch-and-price procedure as soon as its lower bound reaches or exceeds this value.

\section{Experimental setup and Branch-and-Price configuration}
\label{sec:B&Pconfigurations}

In this section, we describe the experimental setup and report preliminary experiments conducted to configure the B\&P algorithm. 
The computational experiments were carried out on a computer equipped with Ubuntu 22.04.5 LTS, using a single thread of an Intel Core i5-3230M CPU at 2.60 GHz and 6 GB of RAM.  
The repository, including the source code, instances, results, and analysis, is available at \url{https://github.com/maurolucci/dpcp}.

\subsection{Instances}

We generated random DPCP instances based on the number of vertices $|V|$, the edge probability $p$, and the cardinalities  $n = |\PP|$ and $m = |\QQ|$. 
First, an Erdős–Rényi graph with parameters $|V|$ and $p$ is generated. 
Then, each vertex is assigned uniformly at random to a part of $\PP$ and of $\QQ$, while preventing  empty parts in the partitions.

For CFCP, three datasets exhibiting different structural properties are considered. 
First, we generate hypergraphs arising from frequency assignment problems with circular coverage areas. 
The generator takes as input the number of antennas ($c$) and their coverage radius ($r$). 
The antennas are placed uniformly at random within a $100 \times 100$ square region. 
The resulting hypergraph has one vertex for each antenna, while each hyperedge represents the set of antennas covering a connected region of the coverage area (see Figure \ref{fig:ex-three-antennas}). 
To identify these regions, we discretize the plane using a regular grid with a spacing of 0.1 between adjacent points, and compute which antennas are reachable from each point of the grid.
% To this end, the generation process takes as parameters the number of antennas ($c$) and the coverage radius of each antenna ($r$).
% The antennas are placed uniformly at random in a square of side length 100.
% The antennas are distributed uniformly at random within a square region of dimensions $100 \times 100$. Let $P = \{p_1, \dots, p_c\} \subset [0, 100]^2$ denote the coordinate positions of the antennas. The coverage area of antenna $i \in \{1, \dots, c\}$ is modeled as a closed disk $D_i = D(p_i, r)$ of radius $r$. 
% The associated hypergraph $\mathcal{H} = (\mathcal{V}, \mathcal{E})$ is constructed by setting the vertex set to represent the antennas, so $\mathcal{V} = \{1, \dots, c\}$. The family of hyperedges $\mathcal{E}$ represents the set of all possible reachability profiles for a client located in the region. Geometrically, the arrangement of the disks $D_1, \dots, D_c$ partitions the plane into a finite set of connected regions (cells). Consequently, a non-empty subset of vertices $E \subseteq \mathcal{V}$ forms a hyperedge in $\mathcal{E}$ if and only if there exists a point $x \in [0, 100]^2$ covered exclusively by the antennas in $E$.
 The other two sets of CFCP instances correspond to
hypergraphs induced by the open and closed neighborhoods of the vertices of some DIMACS instances (\url{https://mat.tepper.cmu.edu/COLOR/instances.html}). 

As shown in \ref{Prop.Equiv.Formulations}, for instances arising from the PCP and the LCP, formulation \setcoveringform\ reduces to the set covering formulations proposed in \cite{Furini2018,Lucci2024}.
Consequently, under the same algorithmic choices and computational environment, experiments on PCP and LCP instances would provide limited additional insight.
%Accordingly, we focus on general DPCP instances  together with CFCP instances, for which no computational results have been reported, to the best of our knowledge. 

In summary, the performance of the algorithms will be computationally evaluated on the following sets of instances:

\begin{description}
    \item[DPCP-ER-1] contains three random DPCP instances for each combination of the following parameter values: $|V| = 150$, $p \in \{0.25,0.5,0.75\}$, $n \in \{0.2|V|, 0.3|V|\}$, and $m \in \{0.2|V|, 0.3|V|\}$. This results in a total of 36 instances.
    \item[DPCP-ER-2] is analogous to DPCP-ER-1, but with $|V| = 110$.
    \item[CFCP-FA] contains five random CFCP instances of frequency assignment for each combination of the following parameter values: $(c,r) \in \{(9,50),(12,35),(18,20)\}$. This results in a total of 15 instances.
    \item[CFCP-ON] contains the CFCP instances associated with the open neighborhoods of the DIMACS instances. In this dataset, isolated vertices are first removed from the DIMACS instances to guarantee feasibility. Only instances whose conflict graphs have at most 1000 vertices are considered, resulting in a total of 9 instances.
    \item[CFCP-CN] is analogous to CFCP-ON, but with closed neighborhoods.
\end{description}

% \item By Remark~\ref{clique-estable}, edges are added to $G$ until each part $P_i \in \PP$ is a clique (this step is executed only at the root node).
% \item Check whether $|\PP|=1$. In this case, the solution is trivial by Remark~\ref{nm_igual1}.
% \item For each $i \in [n]$, if $P_i$ is a singleton, then the construction described in Lemma~\ref{Pcardinal1} is applied to reduce the size of the graph. This rule is applied iteratively until no singleton parts remain in $\PP$. If any part of $\PP$ becomes empty, the instance is infeasible.
% \item \textcolor{red}{Colapso de vertices}
% \item Isolated vertices are removed from $G$ and stored in a list. When a solution is found, these vertices are colored in a postprocessing step, as this does not affect the number of required colors, by Corollary~\ref{singleton}.
% \end{enumerate}

The next two sections present preliminary experiments on the DPCP-ER-1 dataset aimed at fine-tuning the B\&P algorithm. 
%Subsequently, we evaluate the performance of the resulting algorithm on the remaining datasets.

\subsection{Evaluation and selection of DPCP heuristics}

We begin by evaluating the performance of the DPCP heuristics presented in Section \ref{sec:heuristics}: the one-step greedy heuristic, and the two-step greedy heuristic with the three vertex-selection criteria \textsc{Deg}, \textsc{QDeg}, and \textsc{SQDeg}. 

Results are reported in Table \ref{table.results.heur}, where each row summarizes the performance of a heuristic over all instances in the dataset, including the number of instances for which a solution was found, the average objective value, and the average execution time. 
Column ``$\geq$best'' reports the percentage of instances in which the heuristic finds the best solution among all heuristics. Since several heuristics may find the same best solution, column ``best'' breaks ties using the execution time, so exactly one heuristic is counted for each instance.
When a heuristic fails to find a solution, we assign as objective value the minimum cardinality of the two partitions, providing a valid upper bound for feasible instances and ensuring a fair comparison in terms of the average objective value.

\begin{table}
\centering
\small
\begin{tabular}{lrrrrr}
\toprule
heuristic & \#solved & value & time & $\geq$best & best \\
& & & (ms) & (\%) & (\%) \\
\midrule
greedy 1-step & 16 & 22.8 & 3.51 & 33 & 11 \\
greedy 2-step (\textsc{Deg}) & 33 & 10.5 & 1.04 & 36 & 36 \\
greedy 2-step (\textsc{QDeg}) & 31 & 12.0 & 1.29 & 39 & 8 \\
greedy 2-step (\textsc{SQDeg}) & 32 & 10.8 & 1.57 & 67 & 36 \\
\bottomrule
\end{tabular}
\caption{Average results of the greedy heuristics for the DPCP-ER-1 dataset.}
\label{table.results.heur}
\end{table}

The two-step greedy heuristics succeed in finding solutions for nearly all instances. 
In contrast, the one-step greedy heuristic fails to find solutions for roughly half of the instances and takes more than three times as long.
Among the vertex-selection criteria, \textsc{SQDeg} achieves the best objective value in 67\% of the instances, nearly twice as often as any other criterion.
However, \textsc{SQDeg} is 34\% slower than \textsc{Deg} and 18\% slower than \textsc{QDeg}. When execution time is used to break ties, \textsc{Deg} emerges as a competitive alternative, achieving the best performance in 36\% of the instances, tied with \textsc{SQDeg}.
Overall, these results demonstrate that the two-step approach consistently outperforms the one-step approach in terms of both solution quality and computational efficiency.

These results motivated the evaluation of semigreedy variants of the two-step heuristics, using $\alpha = 0.15$ to define the restricted candidate list.
We analyze the average best objective value and execution time as a function of the number of iterations, as shown in Figures \ref{fig.results.semigreedy.value} and \ref{fig.results.semigreedy.time}.

\begin{figure}

\begin{subfigure}{0.49\textwidth}
\centering
\includegraphics[width=\textwidth]{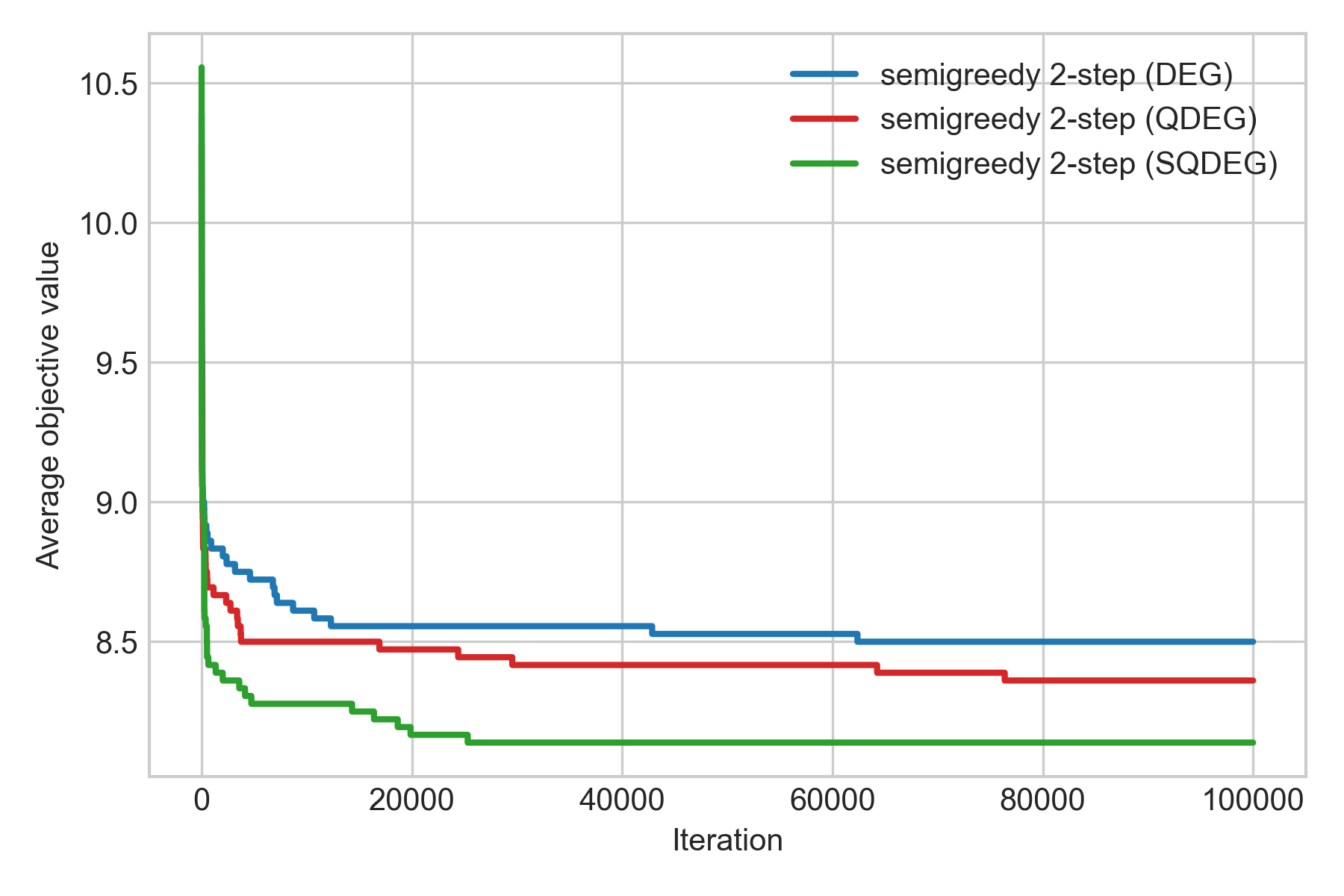}
\caption{Average best objective value}
\label{fig.results.semigreedy.value}
\end{subfigure}
\begin{subfigure}{0.49\textwidth}
\centering
\includegraphics[width=\textwidth]{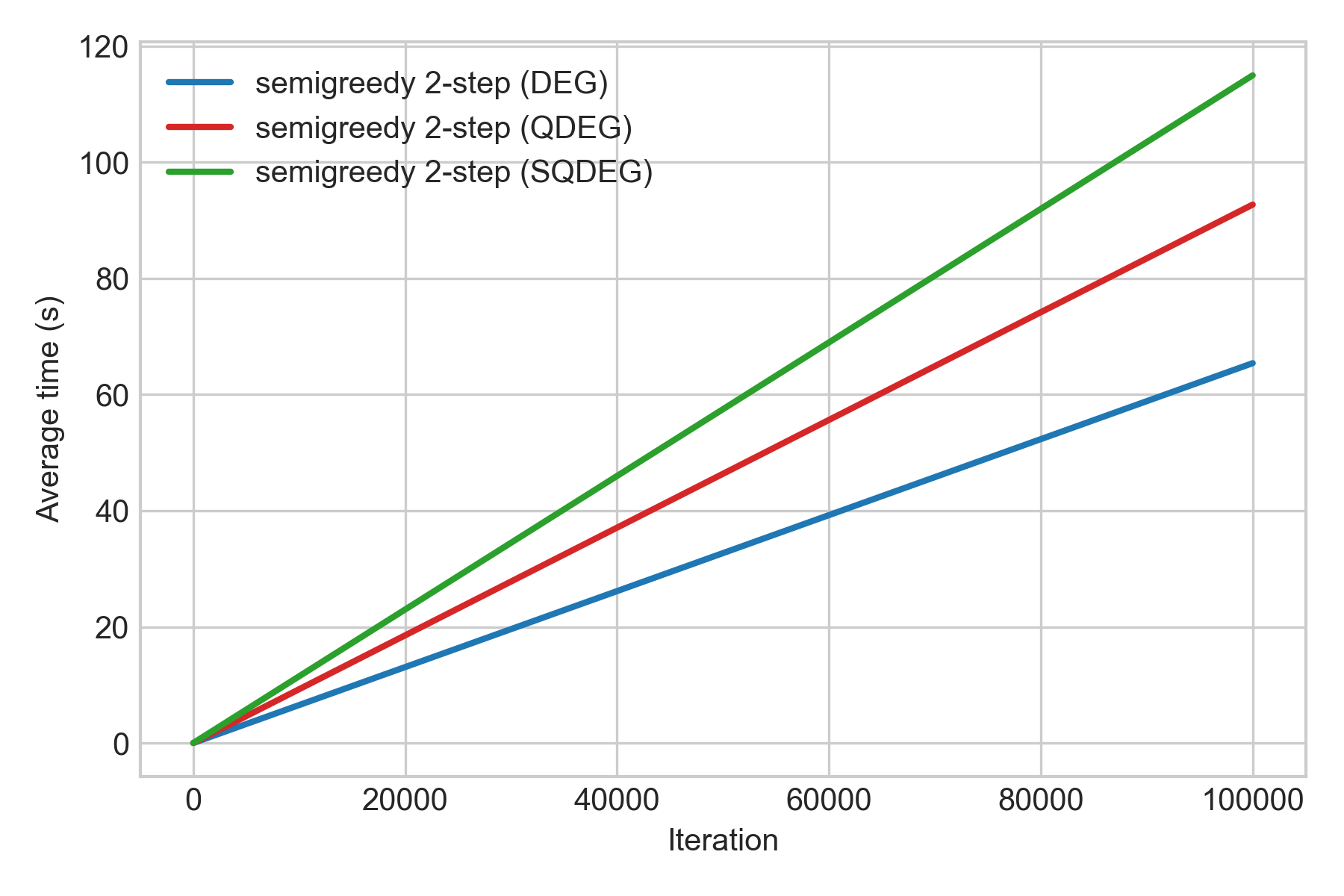}
\caption{Average cumulative execution time}
\label{fig.results.semigreedy.time}
\end{subfigure}

\caption{Results of the semigreedy heuristics for the DPCP-ER-1 dataset.}
\end{figure}

Introducing randomness into the two-step heuristics quickly yields improvements in solution quality of more than 20\% compared with their greedy counterparts.
Among the three criteria, \textsc{SQDeg} consistently provides the best solutions. 
The improvement in solution quality becomes progressively smaller as the number of iterations increases, whereas the running time grows almost linearly with the number of iterations.
For \textsc{SQDeg}, all the improvement is obtained within the first 30,000 iterations.
All heuristics find solutions for the 34 feasible instances of the dataset within the first 300 iterations, while the remaining two instances are correctly identified as infeasible during preprocessing.

Based on these results, we adopt the two-step approach with \textsc{SQDeg}. The greedy version of the heuristic is used to quickly initialize the linear relaxation at every node of the B\&P tree, whereas the semi-greedy version is executed only once at the root node to obtain a high-quality initial upper bound that promotes earlier pruning of the search tree.
The latter runs for 40,000 iterations or 1 minute, whichever comes first.

\subsection{Evaluation and selection of pricing algorithms}

As a second experiment, we evaluate the pricing algorithms presented in Section \ref{sec:branch-and-price-algo-pricing}. 
To this end, we solve the linear relaxation at the root node of each instance, with a time limit of 15 minutes.
We generate the initial columns as described in Section~\ref{sec:branch-and-price-algo-init}.
More specifically, we use a greedy solution obtained by the two-step greedy heuristic using the \textsc{SQDeg} criterion. 
If the heuristic fails to produce a feasible solution, a dummy variable is used instead, with objective coefficient $M=1000$, as suggested in \cite{Lucci2024}.

Our baseline pricing algorithm consists of solving \pricingform\ directly with an ILP solver.
We incorporate an early-stopping mechanism that terminates the search as soon as a stable set of weight at least 1.1 is found.
This avoids the additional effort of proving optimality once a sufficiently attractive column has been identified. 
Starting from this baseline, the remaining pricing algorithms are incorporated incrementally to evaluate the contribution of each enhancement.

First, we enable heuristic pricing by executing the randomized heuristic until either 20 candidate stable sets are found or 100 iterations are reached, with $\alpha=0.2$ for the restricted candidate list. 
If no columns are found, we resort to exact pricing.

Subsequently, the two WSSP-based pricing algorithms are activated, separately and jointly. These algorithms are applied after heuristic pricing fails and before exact pricing is invoked. 
Both rely on the exact algorithm of \cite{Held2012} for WSSP (available at \url{https://github.com/heldstephan/exactcolors}), which combines a greedy heuristic with combinatorial enumeration.
We refer to these pricing algorithms as WSSP$_{\PP}$ and WSSP$_{\PP\QQ}$, respectively, following the notation of their corresponding weight functions.
%For WSSP$_{\PP\QQ}$ pricing, we additionally exploit two features of the implementation: early termination and returning multiple stable sets $S$ with $w_{\PP\QQ}(S) > 1$ in a single invocation. 

In Table \ref{table.results.column.generation}, each row summarizes the average performance of a pricing configuration over all instances in the dataset. 
The first column indicates the configuration, while the second reports the number of linear relaxations solved within the time limit. 
The remaining columns report the average total pricing time and the average total number of generated columns, both broken down by pricing algorithm.
The enabled pricing algorithms for each configuration (i)–(v) can be identified directly from the table, where entries marked ``--'' indicate disabled algorithms.

\begin{table}[]
\centering
\small
\setlength{\tabcolsep}{5pt}
\begin{tabular}{ccrccccccccc}
\toprule
config. & \#solved & \multicolumn{5}{c}{pricing time (s)} & \multicolumn{5}{c}{\#added cols} \\
\cmidrule(lr){3-7} \cmidrule(lr){8-12}
& & total & heur & WSSP$_{\PP}$ & WSSP$_{\PP\QQ}$ & exact & total & heur & WSSP$_{\PP}$ & WSSP$_{\PP\QQ}$ & exact \\
\midrule
i & 22 & 603.1 & -- & -- & -- & 603.1 & 315 & -- & -- & -- & 315 \\
ii & 28 & 449.2 & 1.2 & -- & -- & 448.0 & 551 & 512 & -- & -- & 38 \\
iii & 32 & 279.3 & 1.3 & 1.7 & -- & 276.3 & 570 & 500 & 47 & -- & 21 \\
iv & 29 & 418.4 & 1.2 & -- & 2.7 & 414.5 & 554 & 512 & -- & 6.6 & 35 \\
v & 32 & 259.3 & 1.3 & 1.6 & 1.3 & 255.1 & 571 & 500 & 48 & 1.5 & 20 \\
\bottomrule
\end{tabular}
\caption{Average results of the pricing algorithms on the DPCP-ER-1 linear relaxations.}
\label{table.results.column.generation}
\end{table}

All pricing algorithms improve the baseline, with heuristic pricing generating most of the added columns and reducing those generated by exact pricing from 315 to fewer than 40.
Nevertheless, most of the pricing time is still spent in exact pricing, suggesting that further improvements should focus on this component.
Overall, combining heuristic pricing with both WSSP-based pricing algorithms yields the best performance, reducing the average pricing time by 57\% (from 603.1 s to 259.3 s) while increasing the number of solved instances from 22 to 32. 
Thus, configuration (v) is adopted throughout the B\&P algorithm.

\medskip

Preliminary tests suggested the convenience of keeping a private column pool at each node of the B\&P tree, which 
is scanned before any pricing algorithm is invoked. 
At the root node, the pool is initialized with the stable sets generated during the execution of the semi-greedy two-step heuristic. 
At each non-root node, the set of columns generated during initialization is augmented with the valid columns having positive value in the optimal solution of the parent node.
The remaining inherited columns from the parent node are stored in the node's private pool. 
To limit memory consumption, each pool stores at most 1000 columns.

In what follows, the B\&P tree is explored using a best-bound strategy, selecting for branching the node with the lowest linear relaxation value.
Each node of the tree is preprocessed by  repeatedly applying Lemmas \ref{th:singletons} and \ref{Isolated} until no further reduction is possible. This reduces the instance size and enables early infeasibility detection.
%Once a solution has been found, isolated vertices, if any, are colored in a postprocessing step, since this does not affect the number of required colors (see Lemma \ref{Isolated}).

\section{Computational evaluation of DPCP solvers} \label{sec:computationalresults}

In this section, we evaluate two approaches for solving the DPCP: the B\&P algorithm, using the settings established in the previous section, and the direct solution of the compact ILP formulation proposed in \cite{LucciEtAl2025}.
For the latter, we use Gurobi v12.0.3 with its default parameters. 
The evaluation is conducted on  DPCP-ER-2 dataset and the three datasets containing CFCP instances, with a time limit of one hour per instance.
To ensure a fair comparison, both approaches start from the same preprocessed instance and are initialized with the same solution generated by the two-step semi-greedy heuristic, which serves as an initial incumbent for the B\&P algorithm and as a warm start for the compact formulation.

\subsection{Performance on random DPCP instances}

Table \ref{table.results.byp} presents the results for the DPCP-ER-2 dataset, where each row reports the average performance over instances with the same edge probability $p$.
%whereas in Table \ref{table.results.byp.nm}, the averages are computed over instances with the same cardinalities $n = |\PP|$ and $m = |\QQ|$.
The columns report the solver (B\&P or Gurobi), the total number of instances solved to optimality (``opt''), proven infeasible (``infeas''), or terminated due to the time limit (``tilim''), the average execution time (excluding the time spent by the initial execution of the DPCP heuristic), and the absolute optimality gap.
The latter is averaged over all feasible instances, assigning a gap of 0 to instances solved to optimality.

\begin{table}

% \begin{subtable}{0.49\textwidth}
\centering
\small
\setlength{\tabcolsep}{5pt}
\begin{tabular}{rlrrrrr}
\toprule
$p$ & solver & \multicolumn{3}{c}{\#status} & time & abs.  \\
\cmidrule(lr){3-5}
 &  & opt & infeas & tilim & (s) & gap \\
\midrule
\multirow{2}{*}{0.25} & b\&p & 7 & 1 & 4 & 1281 & 0.4 \\
 & gurobi & 11 & 1 & 0 & 8 & 0.0 \\
\midrule
\multirow{2}{*}{0.50} & b\&p & 8 & 2 & 2 & 1275 & 0.2 \\
& gurobi & 5 & 2 & 5 & 1580 & 0.6 \\
\midrule
\multirow{2}{*}{0.75} & b\&p & 8 & 2 & 2 & 763 & 0.2 \\
& gurobi & 1 & 2 & 9 & 2867 & 2.7 \\
\bottomrule
\end{tabular}
% \caption{By edge probability.}
% \label{table.results.byp.p}
% \end{subtable}
% \begin{subtable}{0.49\textwidth}
% \centering
% \small
% \setlength{\tabcolsep}{5pt}
% \begin{tabular}{rrlrrrrr}
% \toprule
% $n$ & $m$ & solver & \multicolumn{3}{c}{\#status} & time & abs. \\
% \cmidrule(lr){4-6}
%  &  &  & opt & infeas & tilim & (s) & gap \\
% \midrule
% \multirow{4}{*}{22} & \multirow{2}{*}{22} & b\&p & 7 & 2 & 0 & 287 & 0.0 \\
%  &  & gurobi & 5 & 2 & 2 & 874 & 0.4 \\
% \cmidrule(l){2-8}
%  & \multirow{2}{*}{33} & b\&p & 9 & 0 & 0 & 314 & 0.0 \\
%  & & gurobi & 7 & 0 & 2 & 1056 & 0.2 \\
%  \midrule
% \multirow{4}{*}{33} & \multirow{2}{*}{22} & b\&p & 3 & 3 & 3 & 1475 & 0.5 \\
%  &  & gurobi & 2 & 3 & 4 & 1609 & 2.2 \\
% \cmidrule(l){2-8}
%  & \multirow{2}{*}{33} & b\&p & 4 & 0 & 5 & 2349 & 0.6 \\
%  &  & gurobi & 3 & 0 & 6 & 2401 & 1.7 \\
% \bottomrule
% \end{tabular}
% \caption{By partition size.}
% \label{table.results.byp.nm}
% \end{subtable}

\caption{Average results by edge probability on the DPCP-ER-2 dataset.}
\label{table.results.byp}
\end{table}

As shown in the table, the performance is strongly affected by the edge probability $p$. 
While B\&P improves as $p$ increases, the compact formulation exhibits the opposite behavior. 
In particular, B\&P clearly outperforms the compact formulation on high-density instances, proving optimality for 8 instances compared with only 1 when $p = 0.75$. 
Conversely, the compact formulation performs better on low-density instances, solving 11 instances to optimality compared with 7 for B\&P when $p = 0.25$, while requiring only 8 seconds on average versus 1281 seconds for B\&P. 
For medium-density instances ($p = 0.5$), B\&P has a slight advantage over the compact formulation, proving optimality for 8 instances against 5.

Additionally, although not shown in the table, all five infeasible instances in this dataset have $m=0.2|V|$, the smallest value considered, suggesting a higher likelihood of generating infeasible instances for smaller values of $m$. In all cases, infeasibility was detected in the preprocessing phase.
However, as mentioned in Example \ref{example:infeasible:branch}, infeasible nodes may still arise during the B\&P search on feasible instances.
In practice, this occurs very rarely: across all instances, only 334 of the 8,888 generated nodes (3.76\%) required dummy initialization, and only 3 were proven to have an infeasible relaxation (0.03\% of all nodes).
Therefore, additional infeasibility detection procedures are not a priority for this dataset.

% When grouping instances by the partitions size, B\&P solves one to three more instances to optimality than the compact formulation in every $(n,m)$ category, while also achieving lower average running times.
% Both approaches show a degradation in performance as the values of $n$ and $m$ increase.

To further analyze performance, Table \ref{table.results.byp.inst} presents the instance-by-instance results, excluding infeasible instances.
The first column contains the instance name, which encodes the number of vertices $|V|$, the edge probability $p$, and the cardinalities $n$ and $m$ of $\mathcal{P}$ and $\mathcal{Q}$, respectively, and the instance seed (e.g., \texttt{r\_V110\_p0.25\_n22\_m22\_i0}).
The column ``hval'' reports the value of the initial solution obtained by the DPCP heuristic.
For each solver, we report the execution time, the number of generated nodes, and the final lower and upper bounds.
For B\&P, we additionally report the value of the root linear relaxation (``rval''), as well as the pricing time and the number of generated columns, including columns retrieved from the pool. These pricing statistics are reported separately for the root node and averaged over the remaining nodes.
We use boldface to indicate the best result between the two algorithms, with execution time as the primary criterion. In case of a tie at the time limit, the upper bound is used as the secondary criterion, followed by the lower bound.

\begin{sidewaystable}
\centering
\small
\setlength{\tabcolsep}{4pt}
\begin{tabular}{lrlrrrrllrrlrlr}
\toprule
instance & hval & \multicolumn{9}{c}{b\&p} & \multicolumn{4}{c}{gurobi} \\
\cmidrule(lr){3-11} \cmidrule(lr){12-15}
 &  & time (s) & nodes & rval & lb & ub & \multicolumn{2}{c}{pricing time (s)} & \multicolumn{2}{c}{\#added cols} & time (s) & nodes & lb & ub \\
\cmidrule(lr){8-9}\cmidrule(lr){10-11}
 &  &  &  &  &  &  & root & non-root & root & non-root &  &  &  &  \\
\midrule
\texttt{r\_V110\_p0.25\_n22\_m22\_i0} & 2 & 6.6 & 1 & 1.6 & 2.0 & 2 & 6.6 & -- & 420 & -- & $<$\textbf{0.1} & 0 & 2.0 & 2 \\
\texttt{r\_V110\_p0.25\_n22\_m22\_i1} & 2 & 34.9 & 1 & 1.5 & 2.0 & 2 & 34.9 & -- & 474 & -- & \textbf{0.2} & 1 & 2.0 & 2 \\
\texttt{r\_V110\_p0.25\_n22\_m22\_i2} & 3 & 833.1 & 193 & 1.6 & 2.0 & 2 & 12.9 & 4.3 & 505 & 100 & \textbf{0.6} & 1 & 2.0 & 2 \\
\texttt{r\_V110\_p0.25\_n22\_m33\_i0} & 2 & 9.2 & 1 & 1.6 & 2.0 & 2 & 9.1 & -- & 460 & -- & $<$\textbf{0.1} & 0 & 2.0 & 2 \\
\texttt{r\_V110\_p0.25\_n22\_m33\_i1} & 2 & 2.0 & 1 & 1.6 & 2.0 & 2 & 1.9 & -- & 475 & -- & $<$\textbf{0.1} & 0 & 2.0 & 2 \\
\texttt{r\_V110\_p0.25\_n22\_m33\_i2} & 2 & 9.5 & 1 & 1.6 & 2.0 & 2 & 9.5 & -- & 455 & -- & \textbf{0.1} & 1 & 2.0 & 2 \\
\texttt{r\_V110\_p0.25\_n33\_m22\_i1} & 4 & \textit{tilim} & 234 & 2.6 & 3.0 & 4 & 39.3 & 15.2 & 802 & 166 & \textbf{10.4} & 667 & 3.0 & 3 \\
\texttt{r\_V110\_p0.25\_n33\_m22\_i2} & 4 & \textit{tilim} & 119 & 2.6 & 3.0 & 4 & 78.9 & 29.7 & 844 & 168 & \textbf{75.0} & 5804 & 4.0 & 4 \\
\texttt{r\_V110\_p0.25\_n33\_m33\_i0} & 4 & \textit{tilim} & 363 & 2.6 & 3.0 & 4 & 31.1 & 9.8 & 733 & 148 & \textbf{2.8} & 569 & 3.0 & 3 \\
\texttt{r\_V110\_p0.25\_n33\_m33\_i1} & 3 & 77.8 & 1 & 2.4 & 3.0 & 3 & 77.7 & -- & 776 & -- & \textbf{1.7} & 4 & 3.0 & 3 \\
\texttt{r\_V110\_p0.25\_n33\_m33\_i2} & 4 & \textit{tilim} & 304 & 2.5 & 3.0 & 4 & 40.4 & 11.7 & 686 & 144 & \textbf{9.9} & 652 & 3.0 & 3 \\
\midrule
\texttt{r\_V110\_p0.5\_n22\_m22\_i1} & 4 & 765.6 & 97 & 2.8 & 4.0 & 4 & 21.5 & 7.7 & 384 & 78 & \textbf{522.9} & 39595 & 4.0 & 4 \\
\texttt{r\_V110\_p0.5\_n22\_m22\_i2} & 4 & 503.8 & 41 & 2.9 & 4.0 & 4 & 25.6 & 11.9 & 373 & 77 & \textbf{145.0} & 4964 & 4.0 & 4 \\
\texttt{r\_V110\_p0.5\_n22\_m33\_i0} & 4 & 91.5 & 35 & 2.9 & 4.0 & 4 & 9.5 & 2.4 & 346 & 62 & \textbf{21.3} & 3899 & 4.0 & 4 \\
\texttt{r\_V110\_p0.5\_n22\_m33\_i1} & 4 & 899.9 & 359 & 2.8 & 4.0 & 4 & 7.7 & 2.5 & 389 & 63 & \textbf{245.8} & 18697 & 4.0 & 4 \\
\texttt{r\_V110\_p0.5\_n22\_m33\_i2} & 4 & 1216.1 & 419 & 2.7 & 3.0 & 3 & 6.3 & 2.9 & 364 & 64 & \textbf{35.9} & 1280 & 3.0 & 3 \\
\texttt{r\_V110\_p0.5\_n33\_m22\_i1} & 7 & \textbf{2114.4} & 349 & 5.2 & 6.0 & 6 & 27.9 & 6.0 & 468 & 98 & \textit{tilim} & 38099 & 5.0 & 6 \\
\texttt{r\_V110\_p0.5\_n33\_m22\_i2} & 7 & \textit{tilim} & 259 & 5.2 & 6.0 & 7 & 61.6 & 13.6 & 491 & 97 & \textit{tilim} & 73635 & 4.0 & \textbf{6} \\
\texttt{r\_V110\_p0.5\_n33\_m33\_i0} & 6 & \textbf{1326.7} & 243 & 4.7 & 6.0 & 6 & 12.1 & 5.4 & 497 & 88 & \textit{tilim} & 108460 & 5.0 & 6 \\
\texttt{r\_V110\_p0.5\_n33\_m33\_i1} & 6 & \textit{tilim} & 466 & 4.6 & 5.0 & 6 & 33.1 & 7.6 & 460 & 84 & \textit{tilim} & 100742 & 5.0 & 6 \\
\texttt{r\_V110\_p0.5\_n33\_m33\_i2} & 6 & \textbf{1190.7} & 217 & 4.7 & 6.0 & 6 & 10.1 & 5.5 & 510 & 92 & \textit{tilim} & 112387 & 5.0 & 6 \\
\midrule
\texttt{r\_V110\_p0.75\_n22\_m22\_i1} & 6 & \textbf{155.9} & 47 & 4.8 & 6.0 & 6 & 6.8 & 3.2 & 235 & 46 & \textit{tilim} & 120435 & 5.0 & 6 \\
\texttt{r\_V110\_p0.75\_n22\_m22\_i2} & 6 & \textbf{288.2} & 57 & 4.8 & 6.0 & 6 & 16.6 & 4.8 & 287 & 49 & \textit{tilim} & 103374 & 4.0 & 6 \\
\texttt{r\_V110\_p0.75\_n22\_m33\_i0} & 6 & \textbf{236.6} & 145 & 4.7 & 6.0 & 6 & 6.9 & 1.6 & 259 & 33 & \textit{tilim} & 258464 & 5.0 & 6 \\
\texttt{r\_V110\_p0.75\_n22\_m33\_i1} & 6 & \textbf{145.1} & 69 & 4.6 & 5.0 & 5 & 4.4 & 2.1 & 264 & 38 & 2005.0 & 26608 & 5.0 & 5 \\
\texttt{r\_V110\_p0.75\_n22\_m33\_i2} & 6 & \textbf{219.5} & 115 & 4.6 & 5.0 & 5 & 5.7 & 1.9 & 281 & 36 & \textit{tilim} & 116335 & 4.0 & 5 \\
\texttt{r\_V110\_p0.75\_n33\_m22\_i1} & -- & \textbf{18.3} & 15 & 10.8 & 12.0 & 12 & 7.9 & 0.7 & 322 & 75 & \textit{tilim} & 1855 & 8.0 & 12 \\
\texttt{r\_V110\_p0.75\_n33\_m22\_i2} & 12 & \textbf{346.4} & 71 & 9.7 & 11.0 & 11 & 17.5 & 4.7 & 421 & 66 & \textit{tilim} & 10884 & 6.0 & 12 \\
\texttt{r\_V110\_p0.75\_n33\_m33\_i0} & 10 & \textit{tilim} & 2465 & 8.1 & \textbf{9.0} & 10 & 3.2 & 1.4 & 311 & 46 & \textit{tilim} & 18631 & 6.0 & 10 \\
\texttt{r\_V110\_p0.75\_n33\_m33\_i1} & 10 & \textit{tilim} & 1918 & 8.0 & \textbf{9.0} & 10 & 13.2 & 1.8 & 375 & 50 & \textit{tilim} & 7451 & 6.0 & 10 \\
\texttt{r\_V110\_p0.75\_n33\_m33\_i2} & 10 & \textbf{546.1} & 275 & 8.2 & 9.0 & 9 & 10.7 & 1.9 & 345 & 57 & \textit{tilim} & 7406 & 6.0 & 10 \\
\bottomrule
\end{tabular}
\caption{Per-instance results on the DPCP-ER-2 dataset.}
\label{table.results.byp.inst}
\end{sidewaystable}

The individual results confirm the trends previously observed for low- and high-density graphs, and provide further insight into medium-density graphs, where the relative performance of the two approaches depends on $n$.
When $\mathcal{P}$ has fewer parts, the compact formulation is consistently faster than B\&P, whereas when $\mathcal{P}$ has more parts, B\&P solves three instances to optimality compared with none for the compact formulation.

Moreover, these results further support the effectiveness of the DPCP heuristic on random instances  and highlight the benefits of the proposed column management strategy over column pooling and inheritance. Specifically, the initial execution of the DPCP heuristic finds solutions within one unit of the optimum for 30 of the 31 instances, and finds optimal solutions for 16 of them.
Furthermore, on average, pricing at non-root nodes requires 68.6\% less time and generates 81.5\% fewer columns than pricing at the root node.

\subsection{Performance on CFCP instances}

In this section, we evaluate the algorithms on CFCP instances with different structural characteristics, using the CFCP-FA, CFCP-ON, and CFCP-CN datasets, to determine whether additional structural factors, beyond those identified on random DPCP instances, influence algorithmic performance.
Since these instances involve substantially larger conflict graphs, we reduce the limit on successive pricing heuristic calls from 100 to 10. 
In addition, the time limit is set to 15 minutes for CFCP-FA instances to avoid memory issues, while remaining at 1 hour for the other datasets.

Table \ref{table.results.freq} reports the results for each frequency assignment instance in the CFCP-FA dataset. The first column lists the instance name, which encodes the number of antennas $c$, the radius $r$, and the instance seed (e.g., \texttt{freq\_c18\_r20\_i0}).
The following columns describe the associated conflict graph by reporting its number of vertices $|V|$, density $p$, the cardinalities $n$ and $m$ of $\PP$ and $\QQ$, and the number of colors in the initial heuristic solution (``hval''). 
The remaining columns report, for each solver, the execution time (excluding the time spent by the initial execution of the DPCP heuristic), the number of explored nodes, and the final lower and upper bounds. 

\begin{table}[]
\footnotesize
\setlength{\tabcolsep}{4pt}
\centering
\begin{tabular}{lrrrrrllrrrrrr}
\toprule
instance & $|V|$ & $p$ & $n$ & $m$ & hval & \multicolumn{2}{c}{time (s)} & \multicolumn{2}{c}{nodes} & \multicolumn{2}{c}{lb} & \multicolumn{2}{c}{ub} \\
\cmidrule(lr){7-8} \cmidrule(lr){9-10} \cmidrule(lr){11-12} \cmidrule(lr){13-14}
 &  &  &  & &  & b\&p & gurobi & b\&p & gurobi & b\&p & gurobi & b\&p & gurobi \\
\midrule
\texttt{freq\_c18\_r20\_i0} & 257 & 0.22 & 94 & 18 & 4 & \textit{tilim} & \textbf{324.2} & 22799 & 2553 & 3.0 & 4.0 & 4 & 4 \\
\texttt{freq\_c18\_r20\_i1} & 212 & 0.25 & 86 & 18 & 4 & \textbf{24.9} & 86.4 & 503 & 3 & 3.0 & 3.0 & 3 & 3 \\
\texttt{freq\_c18\_r20\_i2} & 288 & 0.47 & 96 & 18 & 4 & \textit{tilim} & \textbf{214.1} & 9166 & 585 & 3.0 & 3.0 & 4 & 3 \\
\texttt{freq\_c18\_r20\_i3} & 231 & 0.27 & 87 & 18 & 4 & \textit{tilim} & \textbf{60.6} & 15707 & 1 & 3.0 & 3.0 & 4 & 3 \\
\texttt{freq\_c18\_r20\_i4} & 265 & 0.33 & 91 & 18 & 4 & \textit{tilim} & \textbf{51.9 }& 15044 & 1 & 3.0 & 3.0 & 4 & 3 \\
\midrule
\texttt{freq\_c12\_r35\_i0} & 570 & 0.65 & 111 & 12 & 5 & \textbf{231.3} & \textit{tilim} & 187 & 1 & 3.0 & 2.0 & 3 & 4 \\
\texttt{freq\_c12\_r35\_i1} & 229 & 0.45 & 68 & 12 & 4 & \textbf{22.2} & 110.2 & 261 & 253 & 3.0 & 3.0 & 3 & 3 \\
\texttt{freq\_c12\_r35\_i2} & 337 & 0.55 & 85 & 12 & 4 & \textbf{27.9} & 656.3 & 117 & 1 & 3.0 & 3.0 & 3 & 3 \\
\texttt{freq\_c12\_r35\_i3} & 449 & 0.58 & 104 & 12 & 4 & \textit{tilim} & \textit{tilim} & 2076 & 1 & 3.0 & 2.0 & 4 & \textbf{3} \\
\texttt{freq\_c12\_r35\_i4} & 269 & 0.46 & 75 & 12 & 4 & \textbf{24.6} & 330.3 & 183 & 904 & 3.0 & 3.0 & 3 & 3 \\
\midrule
\texttt{freq\_c9\_r50\_i0} & 297 & 0.72 & 62 & 9 & 3 & \textbf{0.7} & 707.3 & 3 & 1 & 3.0 & 3.0 & 3 & 3 \\
\texttt{freq\_c9\_r50\_i1} & 244 & 0.68 & 56 & 9 & 3 & 185.7 & \textbf{168.6} & 1265 & 124 & 3.0 & 3.0 & 3 & 3 \\
\texttt{freq\_c9\_r50\_i2} & 205 & 0.69 & 49 & 9 & 4 & \textbf{6.7} & 223.5 & 103 & 46 & 3.0 & 3.0 & 3 & 3 \\
\texttt{freq\_c9\_r50\_i3} & 276 & 0.70 & 62 & 9 & 3 & \textbf{0.6} & 429.0 & 3 & 36 & 3.0 & 3.0 & 3 & 3 \\
\texttt{freq\_c9\_r50\_i4} & 258 & 0.71 & 56 & 9 & 3 & \textbf{2.7} & 194.9 & 21 & 522 & 3.0 & 3.0 & 3 & 3 \\
\bottomrule
\end{tabular}
\caption{Per-instance results on the CFCP-FA dataset.}
\label{table.results.freq}
\end{table}

Overall, the results are consistent with the trends observed for the random DPCP instances. 
The compact formulation remains the preferred approach for low-density graphs, whereas B\&P performs better on medium- and high-density graphs. 
A notable difference is observed for the compact formulation at the largest coverage radius, which generates dense conflict graphs: unlike in the previous experiment, it solves all instances to optimality. 
Nevertheless, for these instances, B\&P still achieves substantially shorter running times in all but one instance.

An interesting characteristic of these instances is their small chromatic number: all but one require only three colors. 
The DPCP heuristic finds an optimal solution in 5 out of the 15 instances, a lower proportion than for the random DPCP instances (16 out of 31, excluding infeasible ones).
This suggests that the heuristic could be further improved for instances with this specific structure.

Tables \ref{table.results.open.neighborhoods} and \ref{table.results.closed.neighborhoods} report the results for the CFCP instances associated with open and closed neighborhoods, respectively. 
They follow the same format as the previous table.
These results also support the observed relationship between graph density and algorithm performance. 
The conflict graphs in these instances have densities ranging from 0.38 to 0.75, corresponding to low-medium to high-medium density graphs. 
In this range, B\&P achieves the best running times in most cases, with the only exceptions being the open-neighborhood instance \texttt{jean} and \texttt{queen6\_6}. 
Both algorithms reach the time limit or memory limit on \texttt{queen7\_7}, although B\&P obtains a stronger lower bound for the open-neighborhood instance.
The DPCP heuristic finds an optimal solution for 4 of the 9 open-neighborhood instances and 3 of the 9 closed-neighborhood instances.
Gurobi explores only a single node in all but one of the instances it solves, highlighting its strong heuristic capabilities, but this is not enough to outperform B\&P in running time.

\begin{table}[]

\begin{subtable}{\textwidth}
\centering
\footnotesize
\setlength{\tabcolsep}{5pt}
\begin{tabular}{lrrrrrllrrllrr}
\toprule
instance & $|V|$ & $p$ & $n$ & $m$ & hval & \multicolumn{2}{c}{time (s)} & \multicolumn{2}{c}{nodes} & \multicolumn{2}{c}{lb} & \multicolumn{2}{c}{ub} \\
\cmidrule(lr){7-8} \cmidrule(lr){9-10} \cmidrule(lr){11-12} \cmidrule(lr){13-14}
 &  &  &  &  &  & b\&p & gurobi & b\&p & gurobi & b\&p & gurobi & b\&p & gurobi \\
\midrule
\texttt{david} & 803 & 0.61 & 78 & 87 & 3 & \textbf{12.5} & 37.2 & 3 & 1 & 2.0 & 2.0 & 2 & 2 \\
\texttt{huck} & 599 & 0.46 & 72 & 74 & 6 & \textbf{16.7} & 622.1 & 11 & 1 & 2.0 & 2.0 & 2 & 2 \\
\texttt{jean} & 495 & 0.38 & 66 & 80 & 4 & 1738.7 & \textbf{306.6} & 1365 & 645 & 3.0 & 3.0 & 3 & 3 \\
\texttt{myciel3} & 40 & 0.39 & 11 & 11 & 2 & $<$0.1 & $<$0.1 & 1 & 1 & 2.0 & 2.0 & 2 & 2 \\
\texttt{myciel4} & 142 & 0.40 & 23 & 23 & 2 & \textbf{0.1} & 1.2 & 1 & 1 & 2.0 & 2.0 & 2 & 2 \\
\texttt{myciel5} & 472 & 0.39 & 47 & 47 & 2 & \textbf{3.1} & 61.0 & 1 & 1 & 2.0 & 2.0 & 2 & 2 \\
\texttt{queen5\_5} & 320 & 0.71 & 25 & 25 & 2 & \textbf{0.5} & 53.2 & 1 & 1 & 2.0 & 2.0 & 2 & 2 \\
\texttt{queen6\_6} & 580 & 0.66 & 36 & 36 & 3 & 2669.1 & \textbf{1161.5} & 41 & 1 & 2.0 & 2.0 & 2 & 2 \\
\texttt{queen7\_7} & 952 & 0.61 & 49 & 49 & 4 & \textit{tilim} & \textit{memlim} & 11 & 1 & \textbf{2.0} & 1.0 & 4 & 4 \\
\bottomrule
\end{tabular}
\caption{Open neighborhoods}
\label{table.results.open.neighborhoods}
\end{subtable}

\medskip

\begin{subtable}{\textwidth}
\centering
\footnotesize
\setlength{\tabcolsep}{5pt}
\begin{tabular}{lrrrrrllrrllrr}
\toprule
instance & $|V|$ & $p$ & $n$ & $m$ & hval & \multicolumn{2}{c}{time (s)} & \multicolumn{2}{c}{nodes} & \multicolumn{2}{c}{lb} & \multicolumn{2}{c}{ub} \\
\cmidrule(lr){7-8} \cmidrule(lr){9-10} \cmidrule(lr){11-12} \cmidrule(lr){13-14}
 &  &  &  &  &  & b\&p & gurobi & b\&p & gurobi & b\&p & gurobi & b\&p & gurobi \\
\midrule
\texttt{david} & 791 & 0.63 & 74 & 87 & 3 & \textbf{155.5} & 330.4 & 5 & 1 & 2.0 & 2.0 & 2 & 2 \\
\texttt{huck} & 455 & 0.52 & 49 & 74 & 4 & \textbf{8.2} & 123.4 & 7 & 1 & 2.0 & 2.0 & 2 & 2 \\
\texttt{jean} & 477 & 0.40 & 66 & 80 & 4 & \textbf{93.5} & 164.2 & 55 & 1 & 2.0 & 2.0 & 2 & 2 \\
\texttt{myciel3} & 51 & 0.56 & 11 & 11 & 2 & $<$\textbf{0.1} & 0.1 & 1 & 1 & 2.0 & 2.0 & 2 & 2 \\
\texttt{myciel4} & 165 & 0.50 & 23 & 23 & 3 & \textbf{3.5} & 7.2 & 19 & 1 & 2.0 & 2.0 & 2 & 2 \\
\texttt{myciel5} & 519 & 0.45 & 47 & 47 & 3 & \textbf{202.7} & 239.1 & 49 & 1 & 2.0 & 2.0 & 2 & 2 \\
\texttt{queen5\_5} & 345 & 0.75 & 25 & 25 & 2 & \textbf{0.6} & 66.3 & 1 & 1 & 2.0 & 2.0 & 2 & 2 \\
\texttt{queen6\_6} & 616 & 0.69 & 36 & 36 & 2 & \textbf{6.1} & 432.1 & 1 & 1 & 2.0 & 2.0 & 2 & 2 \\
\texttt{queen7\_7} & 1001 & 0.63 & 49 & 49 & 4 & \textit{tilim} & \textit{memlim} & 1 & 1 & -- & \textbf{1.0} & 4 & 4 \\
\bottomrule
\end{tabular}
\caption{Closed neighborhoods}
\label{table.results.closed.neighborhoods}
\end{subtable}

\caption{Per-instance results on the CFCP-ON and CFCP-CN datasets}

\end{table}

\section{Conclusions and future work}
\label{sec:conclusions}

The Double-Partition Coloring Problem (DPCP) was introduced in \cite{LucciEtAl2025} as a generalization of several coloring problems on graphs and hypergraphs. In this work, we complete the proofs establishing that the List Coloring Problem (LCP) and the Conflict-Free Coloring Problem (CFCP) reduce to the DPCP. In addition, we develop the details of a branch-and-price (B\&P) algorithm for solving a set covering formulation of the DPCP. Together with the compact formulation presented in our previously mentioned work, which is solved by a state-of-the-art IP solver in its standard configuration, these exact approaches provide an effective framework for addressing a wide range of coloring problems. To the best of our knowledge, they are the first exact methods to be proposed and computationally tested on CFCP instances.

The development of a B\&P scheme revealed several structural properties of the DPCP, two of which are worth highlighting here. First, the feasibility of a DPCP instance is characterized in terms of the stability number of a polynomial-sized auxiliary graph. %: an instance $(G,\PP,\QQ)$ is feasible if and only if $\alpha(\tilde G)\geq|\PP|$ (Proposition~\ref{Gtilde}).} 
Second, general DPCP instances cannot be solved independently on each connected component of the underlying graph. We characterize the instances for which such a decomposition is possible.

Greedy and semi-greedy heuristics for the DPCP were proposed following one-step and two-step schemes. 
%Both approaches iteratively construct a solution. In the one-step scheme, a vertex is selected and colored at each iteration, whereas in the two-step scheme, a vertex is selected at each iteration in the first phase, and all selected vertices are colored in the second phase. 
%As usual, our B\&P algorithm performs three main steps at each node of the search tree: initialization, pricing, and branching.
%The initialization may comprehend up to three phases: (i) primal one-step and two-step heuristics are run (with different configurations at the root node and at the other nodes) in order to try to find feasible solutions of good quality; 
%(ii) in case the heuristics fail, we try to solve the instance of the MSSP described in Proposition~\ref{Gtilde} with a prescribed time limit, to either obtain a feasible solution or prove infeasibility of the DPCP instance; (ii) in case the heuristics fail, we resort to computing a basic feasible solution of $\setcoveringformLRM$, a parameterized linear relaxation of \setcoveringform, followed by the application of Remark~\ref{remark.M}  to either obtain a feasible start basis for the column generation algorithm or prove infeasibility of the DPCP instance; (iii) finally, in  case this latter approach fails, too, we proceed with a standard two-phase solution of the linear relaxation of \setcoveringform. However, the third step has not been reached in any of the computational tests. 
%A one-step heuristic and three two-step heuristics with different vertex-selection criteria (\textsc{Deg}, \textsc{QDeg}, and \textsc{SQDeg}) were tested.
In our experimental setting, the two-step heuristics outperformed the one-step heuristic, finding good solutions for nearly all feasible instances with average running times below 2 ms. Introducing randomness improved the solution values by more than 20\% within a few seconds. When used to initialize B\&P, one semi-greedy two-step heuristic found solutions within one unit of the optimum for 30 of 31 feasible instances and optimal solutions for half of these instances. These results highlight the effectiveness of the proposed heuristics both as stand-alone methods and as initialization procedures for B\&P.

%{\color{red} decir algo que la fase 1, desarrollo de heuristicas, puede ser considerado una herramienta mas para resolver en forma aproximada instancias de ...y decir que tienen buena performance ya que alcanzan la optima en tabto por ciento y el gap es a lo sumo tanto?}

%{\color{red} It is proved that} 
%The pricing problem is NP-hard, as it can be reduced from the decision version of the WSSP. 
After showing that the pricing problem is NP-hard, 
we propose an integer programming formulation and three heuristic solution strategies. 
%On one hand, we devise a semi-greedy randomized heuristic based on constructing a stable set by repeatedly selecting a vertex of maximum weight and updating the weights of the remaining vertices. On the other hand, we propose two heuristics based on bounding the objective function of \pricingform\ (from above and below) with functions that can be expressed as sums of (constant) vertex weights. Thus, the solution of the corresponding instances of the WSSP may either generate new columns or provide evidence for safely stopping the column generation process.}
%The linear relaxations of all instances were solved via column generation in order to test five different configurations for addressing the pricing subproblem: the exact solution of \pricingform\ with the CPLEX solver, as baseline, followed by the incremental activation of the semigreedy heuristic, and the two WSSP-based pricing algorithms, running either separately or jointly. 
In the computational experiments, the best overall performance was obtained when all three pricing heuristics were activated. 

The performance of the B\&P algorithm was compared with the direct solution of the compact formulation (using a state-of-the-art solver) on one set of randomly generated DPCP instances and on three sets of CFCP instances (frequency assignment, open neighborhoods and closed neighborhoods). The results show that both approaches complement each other on random DPCP instances and frequency assignment instances: the B\&P tends to perform better when the edge density of the graph is high, while the compact formulation performs better on low-density graphs. For CFCP instances arising from open and closed neighborhoods in graphs, 
%which have low-to-medium and medium-to-high densities and larger numbers of vertices, 
the B\&P algorithm outperforms the compact formulation. 
%Taken together, the two approaches provide an effective framework for solving a wide range of coloring problems, including the Conflict-Free Coloring Problem, for which, to the best of our knowledge, no previous computational studies have been reported.

%Moreover, t
These approaches could be extended to effectively solve a broader class of DPCP instances through further work on specific aspects of the B\&P algorithm.

%{\color{red} FAlta una conclusion fuerte de que proveemos la primer herramienta exacta para resolver CFCP, mas aun, para el DPCP que generaliza y que posibles esudios mas profundos pueden derivar en una herramienta eficiente unica para todos los problemas que generaliza.

%En partiuclar, mencionamos algunos de los aspectos que pueden ser estudiados/analizados. }

Improving the branching scheme could have a significant impact on the performance of the algorithm. 
%Instead of resorting to a VCP solver when $\card{P}=1$ holds for all $P \in \PP$ in a fractional solution at a B\&P node, a second branching rule can be implemented, along the lines proposed in \cite{Furini2018}.
One possible research direction is to introduce an alternative branching rule, as proposed by \citet{Furini2018} for the PCP, which covers the cases 
%$\card{P}=1$ for all $P \in \PP$
when all parts in the first partition of the instance have cardinality equal to one, where our current approach resorts to solving a VCP instance with a separate B\&P procedure.
Moreover, a non-dichotomous rule combining our current rule with the classical VCP branching rule could be explored as an alternative.
 
A further research direction consists of improving the pricing procedure. This could be achieved, for instance, by strengthening the pricing ILP formulation or developing a combinatorial enumeration algorithm, as proposed by \citet{Held2012} for the VCP.

\color{black}

\appendix

\section{\setcoveringform\ vs. extended formulations for PCP and LCP}

In this appendix, we characterize the relationship between \setcoveringform\ and previously proposed set covering formulations for the PCP and the LCP.

\begin{proposition}\label{Prop.Equiv.Formulations}
\setcoveringform\ is equivalent to the formulation $\mathrm{ILP}^{\mathrm{E}}$ proposed in \cite{Furini2018} for the PCP, and to the formulation WLCP-SC proposed in \cite{Lucci2024} for the LCP, in the following sense:
\begin{enumerate}[label=(\alph*)]
    \item \label{prop:formulations:a}
    Let $I=(G,\PP)$ be an instance of the PCP, with $G=(V,E)$, and let $I'=(G,\PP,\QQ)$ be the corresponding DPCP instance with $\QQ=\{\{v\}:v\in V\}$. Then, formulation $\mathrm{ILP}^{\mathrm{E}}$ for $I$ coincides with formulation \setcoveringform\ for $I'$.

    \item \label{prop:formulations:b}
    Let $I=(G,C,\{L_v:v\in V\})$ be an instance of the LCP, with $G=(V,E)$, and assume that $I$ contains no indistinguishable colors, i.e., that the sets
    $V^k=\{v\in V:k\in L_v\}$
    are distinct for all $k\in C$. Let $I'$ be the corresponding DPCP instance constructed as described in Section~\ref{sec:definition:dpcp}. Then, formulation WLCP-SC for $I$ coincides with formulation \setcoveringform\ for $I'$, provided that all color weights are equal to one.
\end{enumerate}
\end{proposition}

\begin{proof}
For part~\ref{prop:formulations:a}, observe that $|Q|=1$ holds for every $Q\in\QQ$ in instance $I'$. Therefore, the family of constraints~\eqref{SCFormulationB2} is empty, and \setcoveringform\ coincides with the formulation $\mathrm{ILP}^{\mathrm{E}}$ proposed in \cite{Furini2018}.

For part~\ref{prop:formulations:b}, following the construction introduced in Section~\ref{sec:definition:dpcp}, let $I'$ denote the DPCP instance $(\tilde G,\PP,\QQ)$, where $\tilde G=(\tilde V,\tilde E)$ is the graph with vertex set
$\tilde V=\{(v,k):v\in V,\ k\in L_v\}$,
and edge set $\tilde E$ such that $(v_1,k_1)(v_2,k_2)\in\tilde E$ if and only if either $v_1v_2\in E$ or $k_1\neq k_2$. Moreover, 
$\PP = \{P_v:v\in V\}$ and  
$\QQ = \{Q_k:k\in C\}$, where
$P_v=\{(v',k)\in\tilde V: v'=v\}$ and 
$Q_k=\{(v,k')\in\tilde V: k' = k\}$.

For each $k\in C$, let $G_k=G[V^k]$ be the subgraph of $G$ induced by $V^k$.
Since $I$ has no indistinguishable colors, the set $K$ of equivalence classes introduced in \cite{Lucci2024} coincides with $C$, and $m(k)=1$ for every $k\in C$. Consequently, the decision variables of WLCP-SC are indexed by
$X=\{(S,k):k\in C,\ S\in\STAB(G_k)\}$.
We first establish a bijection between $X$ and $\STAB(\tilde G)$.

Let $\tilde S\in\STAB(\tilde G)$ with $\tilde S\neq\emptyset$. Then there exists a unique $k\in C$ such that $\tilde S\subseteq Q_k$. Indeed, if $(v_1,k_1),(v_2,k_2)\in\tilde S$ with $k_1\neq k_2$, then $(v_1,k_1)$ and $(v_2,k_2)$ are adjacent in $\tilde G$, contradicting the stability of $\tilde S$. Define $S=\{v\in V:(v,k)\in\tilde S\}$.
Then $S\subseteq V^k$ and $S\in\STAB(G_k)$, since adjacency in $G_k$ would imply adjacency in $\tilde G$. Hence, $(S,k)\in X$.

Conversely, let $(S,k)\in X$ and define
$
\tilde S=\{(v,k):v\in S\}
$.
%\textcolor{blue}{Note that each $\tilde{S}$ is unique because $I$ contains no indistinguishable colors.} 
Then $\tilde S\subseteq\tilde V$, and $\tilde S\in\STAB(\tilde G)$, since adjacency in $\tilde G$ would imply adjacency in $G_k$, contradicting the stability of $S$.

Thus, there is a bijection between the variables $x_{\tilde S}$ in \setcoveringform\ and the variables $x_S^k$ in WLCP-SC.
Moreover, under this correspondence, the condition $\tilde S\cap P_v\neq\emptyset$ is equivalent to $v\in S$, so constraints~\eqref{SCFormulationA} correspond to constraints~(1) of WLCP-SC. Similarly, the condition $\tilde S\cap Q_k\neq\emptyset$ is equivalent to $S\in\STAB(G_k)$, and since $m(k)=1$, constraints~\eqref{SCFormulationB} correspond to constraints~(2) of WLCP-SC.

Finally, when all color weights in WLCP-SC are equal to one, the objective functions of the two formulations coincide.
\end{proof}

\section*{Declaration of generative AI and AI-assisted technologies in the manuscript preparation process}

During the preparation of this work the authors used ChatGPT and Claude in order to improve language and readability. After using this tool/service, the authors reviewed and edited the content as needed and take full responsibility for the content of the published article.

\bibliographystyle{model5-names} 
\bibliography{bibliography}

\newpage

%%%%%%%%%%%%%%%%%%%%%%%%%%%%%%%%%%%%%%%%%%%%%%




% %% For numbered reference style

%   \textit{\LaTeX: a document preparation system},
%   Addison Wesley, Massachusetts,

% \end{thebibliography}

\end{document}